\documentclass[11pt,reqno]{amsart}

\usepackage{dsfont}
\usepackage[bottom]{footmisc}
\usepackage{amssymb, amsmath, amsthm}
\usepackage{enumerate,color}
\newtheorem{thm}{Theorem}[section]
\newtheorem{lem}[thm]{Lemma}

\newtheorem{prop}[thm]{Proposition}
\newtheorem{defn}[thm]{Definition}
\newtheorem{rmk}{Remark}

\numberwithin{equation}{section}

\newcommand{\bel}{\begin{equation} \label}
\newcommand{\ee}{\end{equation}}
\def\beq{\begin{equation}}
\def\eeq{\end{equation}}
\newcommand{\bea}{\begin{eqnarray}}
\newcommand{\eea}{\end{eqnarray}}
\newcommand{\beas}{\begin{eqnarray*}}
\newcommand{\eeas}{\end{eqnarray*}}
\newcommand{\pd}{\partial}

\newcommand{\dd}{\mbox{d}}

\newcommand{\QQQ}{(0,T)\times \Omega}

\newcommand{\R}{\mathbb{R}}

\newcommand{\cA}{\mathcal{A}}

\newcommand{\cC}{\mathcal{C}}

\newcommand{\cL}{\mathcal{L}}

\allowdisplaybreaks

\def\epsilon{\varepsilon}
\def\phi {\varphi}

\providecommand{\abs}[1]{\left\lvert#1\right\rvert}
\providecommand{\norm}[1]{\left\lVert#1\right\rVert}

\renewcommand{\leq}{\leqslant}
\renewcommand{\geq}{\geqslant}

\providecommand{\abs}[1]{\left\lvert#1\right\rvert}
\providecommand{\norm}[1]{\left\lVert#1\right\rVert}

\title[Well-posedness for weak and strong solutions]
{\bf Well-posedness for weak and strong solutions of non-homogeneous
initial boundary value problems for fractional diffusion equations }

\author{
Yavar Kian$^1$,
and Masahiro Yamamoto$^{2,3,4}$
}

\date{}

\begin{document}

\begin{abstract} 
We study the well-posedness for initial boundary value problems 
associated with time fractional diffusion equations with non-homogenous 
boundary and initial values.  We consider both weak and strong solutions
for the problems. For weak solutions, we introduce a new definition of  
solutions which allows to prove the existence of solution to the initial 
boundary value problems with non-zero initial and boundary values and 
non-homogeneous terms lying in some arbitrary negative-order Sobolev spaces. 
For strong solutions, 
we introduce an optimal compatibility condition and
prove the existence of the solutions. 
We introduce also some sharp conditions guaranteeing the
existence of solutions with more regularity in time and space.

\end{abstract}

\maketitle

\renewcommand{\thefootnote}{\fnsymbol{footnote}}
\footnotetext{\hspace*{-5mm} 
\begin{tabular}{@{}r@{}p{13cm}@{}} 
\\
$^1$& Aix Marseille Universit\'e, Universit\'e de Toulon, CNRS, CPT, Marseille,
France; Email: yavar.kian@univ-amu.fr\\
$^2$& Graduate School of Mathematical Sciences, 
the University of Tokyo, 
3-8-1 Komaba, Meguro-ku, Tokyo 153-8914, Japan; \\
$^3$&Honorary Member of Academy of Romanian Scientists, Splaiul Independentei
Street, no 54, 050094 Bucharest Romania;\\
$^4$&Peoples' Friendship University of Russia (RUDN University) 
6 Miklukho-Maklaya St, Moscow, 117198, Russia;
E-mail:
myama@next.odn.ne.jp
\\

 \end{tabular}}


\section{Introduction}

\subsection{Settings}
\label{sec-settings}

Let $\Omega$ be a bounded and connected open subset of $\R^d$, $d \geq 2$, 
with $\mathcal C^2$  boundary $\partial \Omega$. 
Let $a:=(a_{i,j})_{1 \leq i,j \leq d} \in \cC^1(\overline{\Omega};\R^{d^2})$
be symmetric, that is 
$$ 
a_{i,j}(x)=a_{j,i}(x),\ x \in \Omega,\ i,j = 1,\ldots,d, 
$$
and fulfill the ellipticity condition: there exists a constant
$c>0$ such that
\bel{ell}
\sum_{i,j=1}^d a_{i,j}(x) \xi_i \xi_j \geq c |\xi|^2, \quad 
\mbox{for each $x \in \overline{\Omega},\ \xi=(\xi_1,\ldots,\xi_d) \in \R^d$}.
\ee
Assume that $q \in L^\infty(\Omega)$ satisfies 
\bel{a9}
\mbox{there exists a constant $q_0>0$ such that
$q(x)\geq q_0$ for $x \in \Omega$},
\ee
and define the operator $\cA$ by
$$ 
\cA u(x) :=-\sum_{i,j=1}^d \partial_{x_i} 
\left( a_{i,j}(x) \partial_{x_j} u(x) \right)+q(x)u(x),\  x\in\Omega. 
$$ 
Throughout the article, we set 
$$
Q:= (0,T) \times \Omega, \quad \Sigma := (0,T) \times \partial\Omega.
$$
Next, for $\alpha\in(0,1)\cup(1,2)$, $T\in(0,+\infty)$ and
$\rho \in L^\infty(\Omega)$ obeying
\bel{eq-rho}
 0<\rho_0 \leq\rho(x) \leq\rho_M <+\infty,\ x \in \Omega, 
\ee
we consider the following initial boundary value problem (IBVP):
\bel{eq1}
\left\{ \begin{array}{rcll} 
(\rho(x) \partial_t^{\alpha}+\cA ) u(t,x) & = & F(t,x), & (t,x)\in 
Q,\\
\tau_{\chi} u(t,x) & = & f(t,x), & (t,x) \in \Sigma, \\  
\pd_t^k u(0,x) & = & u_k, & x \in \Omega,\ k=0,\ldots,\lceil\alpha\rceil-1,
\end{array}
\right.
\ee
where $\chi=0,1$, $\lceil\cdot\rceil$ denotes the ceiling function:
$$
\lceil\alpha\rceil = 
\left\{ \begin{array}{rl}
& 1 \quad \mbox{if $0<\alpha<1$,}\\
& 2 \quad \mbox{if $1<\alpha<2$,}
\end{array}\right. 
$$
and $\partial_t^\alpha$ denotes the fractional Caputo derivative of order 
$\alpha$ with respect to $t$, defined by
\bel{cap} 
\pd_t^\alpha u(t,x):=\frac{1}{\Gamma(\lceil\alpha\rceil-\alpha)}\int_0^t(t-s)^{\lceil\alpha\rceil-1-\alpha}\pd_s^{\lceil\alpha\rceil} u(s,x) \dd s,\ (t,x) 
\in Q.
\ee
Here the boundary operators $\tau_{\chi}$, $\chi=0,1$, are defined by:
\begin{enumerate}[(a)] 
\item $\tau_0 u:=u$,
\item $\tau_1 u:= \partial_{\nu_\mathcal A} u$, where $\partial_{\nu_\mathcal A}$ stands for the normal derivative with respect to 
$a=(a_{i,j})_{1 \leq i,j \leq d}$, and is given by
$$\partial_{\nu_\mathcal A }h(x) := \sum_{i,j=1}^d a_{i,j}(x) \partial_{x_j} h(x) \nu_i(x),\  x \in \partial\Omega, $$
and $\nu=(\nu_1,\ldots,\nu_d)$ is the outward unit normal vector to $\partial\Omega$. 
\end{enumerate}

\begin{rmk}
We can omit the condition \eqref{a9} but for simplicity, we assume it.
All the results of this article can be easily extended to Robin boundary 
conditions. 
Moreover, applying the fixed point argument, one can extend 
our results to a more general equation:
$$
\partial_t^\alpha u+\mathcal Au +B(t,x)\cdot\nabla_x u+V(t,x)u=0,
$$
where some suitable assumptions are imposed on the coefficients 
$B\in L^\infty(Q)^d$, $V\in L^\infty(Q)$. 
However, in this article, in order to avoid the inadequate expense of the
size of descriptions, 
we do not consider such extensions of our results.\end{rmk}

In the present article, we study the well-posedness for problem \eqref{eq1} 
in a strong and a weak senses. In the weak sense we will prove the well-posedness of \eqref{eq1} when data $f$ and  $u_0$, $u_{\lceil\alpha\rceil-1}$ are lying in some negative-order Sobolev spaces. The strong solution of  \eqref{eq1} 
corresponds to smooth solutions of this problem in time and space.

\subsection{Motivations and a short bibliographical review}

Recall that the initial boundary value problem \eqref{eq1} is often used for describing anomalous diffusion for several physical phenomenon    such as diffusion of substances in heterogeneous media, 
diffusion of 
fluid flow in inhomogeneous anisotropic porous media,  diffusion of carriers in amorphous photoconductors, diffusion in a
turbulent flow (see e.g., \cite{CSLG}). The case $\alpha\in(0,1)$ 
corresponds to a subdiffusive model, while the case $\alpha\in(1,2)$  
corresponds to a super diffusive case. 

The well-posdness for the problem \eqref{eq1} has been intensively studied these last decades.  Many authors considered problem \eqref{eq1} 
for $\alpha\in(0,1)$ with $f=0$ or $\Omega=\R^d$. 
For $\Omega=\R^d$, one can refer to 
\cite{EK} where the existence of classical solutions of \eqref{eq1} is proved by mean of a representation formula involving the Green  functions for \eqref{eq1}. This formulation of the problem has been extended by \cite{Z1}, who proposed a variational formulation of \eqref{eq1} in some abstract 
framework allowing to consider elliptic operators $\mathcal A$ depending on 
$t\in(0,T)$. Note that the 
result of \cite{Z1} can be applied to a bounded domain and the full space 
$\R^d$.
For the case $\Omega = \R^d$ or the half space, in some $L^p$-space, the 
works \cite{DK1,DK2} proved the existence of 
strong solutions of \eqref{eq1} with zero initial data and non-zero 
source terms for elliptic operators $\mathcal A$ whose coefficients 
depend on time.  
For $\Omega=\R^d$, the article \cite{ACV} proved that 
given $u_0$ and bounded $F$, the solutions to \eqref{eq1} is 
H\"older continuous in time and space.  
Similar results were proved in \cite{Z2} in a bounded domain 
with both non-zero source term and initial condition. 
We refer also to \cite{KuY} as for \eqref{eq1} with a time dependent 
elliptic operators $\mathcal A$, where the authors applied the approach of 
\cite{Z1} to establish the existence of strong solutions under 
suitable assumptions. 
In \cite{KY1,SY}, the authors proved the existence of solutions to 
\eqref{eq1} in a bounded domain by mean of an eigenfunction 
representation involving the Mittag-Leffter functions. 
The definition of solutions of \cite{KY1} can be formulated in terms of Laplace transform in time of the solutions. 
On the basis of this last definition, the works \cite{KSY,LKS} proved
the existence of solutions of more complex equations than \eqref{eq1}, 
including fractional diffusion equation with distributed and variable order. 
Moreover, we can refer to the monograph \cite{KRY}.

All the above mentioned results considered \eqref{eq1} with the
homogenous boundary conditions or $\Omega=\R^d$. 
These results can be classified into two categories:
the existence of solutions of \eqref{eq1} in a weak sense and a strong sense,
according to the cases where data belong to some negative-order Sobolev spaces 
and to smoother spaces respectively.
We can refer to the series of works \cite{Ke1,KR} where 
the authors proved the existence of solutions to \eqref{eq1} for 
$u_0=u_{\lceil\alpha\rceil-1}=0$ and $f$ lying in some negative-order 
Sobolev space in time and space, when 
$\mathcal A=-\Delta$ and $\chi=0$ (i.e., the Dirichlet boundary 
condition).  The proof is based on a single layer 
representation of solutions, and an application of some results 
(e.g., \cite{EK,S}) concerning the Green  functions 
for fractional diffusion equations with constant coefficients.
In \cite{Ke2}, the author extended this approach to \eqref{eq1}
for $\alpha\in(0,1)$ with non-homogeneous Robin boundary condition and 
non-vanishing initial condition lying in some H\"older spaces. 
The article \cite{Y} proved the unique existence of weak solution to 
\eqref{eq1} by the transposition (e.g., Lions and Magenes \cite{LM1}) 
when $u_0=0$, $F=0$, $\alpha\in(0,1)$ and 
$f\in L^2(0,T;H^{-\theta}(\partial\Omega))$ with $\theta>0$.  
As for strong solutions, we refer to the work \cite{Gu} 
where the author proved the existence of strong solution $u$ 
to \eqref{eq1} such that $u$ is continuous in time, belongs to 
a class $\mathcal C^2$ in space and $\partial_t^\alpha u$ is 
H\"older continuous in time and space.  

We remark that the well-posdness for problem \eqref{eq1} with 
non-homogenous boundary conditions is meaningful also for other 
mathematical problems such as optimal control problems (see e.g., \cite{Y}) or 
inverse problems (see e.g., \cite{FK,JR,KLLY,KOSY,KY2}).

\subsection{Definition of solutions}
In this subsection we introduce a definition of solutions of \eqref{eq1} for $f\in L^1(0,T;H^{-\theta}(\partial\Omega))$ with $\theta\in\left[\frac{1}{2},
+\infty\right)$ and $(u_0,u_{\lceil\alpha\rceil-1}, F)$ in some negative-order 
Sobolev spaces. To give a suitable definition of such solutions, we define  
the operators $A_\chi$, $\chi=0,1$, in $L^2(\Omega;\rho dx)$ by 
$$
A_{\chi}u = \rho^{-1}\mathcal A u, \quad 
D(A_{\chi}) = \{g\in H^1(\Omega):\ \rho^{-1}\mathcal Ag\in L^2(\Omega),\ 
\tau_{\chi}g_{|\pd\Omega}=0\}.
$$
In view of our assumptions, we have
$D(A_{\chi})=\{g\in H^2(\Omega):\  \tau_\chi g_{|\pd\Omega}=0\}$. Recall that 
the operators $A_{\chi}$, $\chi=0,1$ 
are strictly positive self-adjoint operators 
with compact resolvent. Therefore, for $\chi=0,1$, the spectrum of $A_{\chi}$ 
consists of a non-decreasing sequence of strictly positive eigenvalues 
$(\lambda_{\chi,n})_{n\geq1}$.   Here and henceforth we number 
$\lambda_{\chi,n}$ with the multiplicities for $\chi=0,1$.
In the Hilbert space $L^2(\Omega;\rho dx)$, we introduce an orthonormal basis 
of eigenfunctions $(\phi_{\chi,n})_{n\geq1}$ of $A_{\chi}$ associated 
with the eigenvalues $(\lambda_{\chi,n})_{n\geq1}$. 
From now on, by $\left\langle \cdot, \cdot\right\rangle$ we denote
the scalar product in $L^2(\Omega;\rho dx)$ and 
we set $\mathbb N = \{1,2,\ldots\}$. 
Let us observe that according to the condition imposed on $\rho$, we have 
$L^2(\Omega;\rho dx)=L^2(\Omega)$ with the equivalent norms. 
For all $s\geq 0$, we denote by $A_{\chi}^s$ the operator defined by 
\[
A_{\chi}^s g=\sum_{n=1}^{+\infty}\left\langle g,\phi_n\right\rangle 
\lambda_{\chi,n}^s\phi_{\chi,n},\quad g\in D(A_{\chi}^s)
= \left\{h\in L^2(\Omega):\ \sum_{n=1}^{+\infty}\abs{\left\langle g,
\phi_{\chi,n}\right\rangle}^2 \lambda_{\chi,n}^{2s}<\infty
\right\}
\]
and in $D(A_{\chi}^s)$ we introduce the norm
\[\|g\|_{D(A_{\chi}^s)}
= \left(\sum_{n=1}^{+\infty}\abs{\left\langle g,
\phi_{\chi,n}\right\rangle}^2 \lambda_{\chi,n}^{2s}\right)^{\frac{1}{2}},
\quad g\in D(A_{\chi}^s).
\]
We define $D(A_{\chi}^{-s}) = D(A_{\chi}^s)'$ by the dual space 
to $D(A_{\chi}^s)$, which is a Hilbert space with 
the norm 
\[
\norm{g}_{D(A_{\chi}^{-s})}
= \left(\sum_{n=1}^\infty \abs{\left\langle g,\phi_{\chi,n}\right\rangle
_{-s}}^2\lambda_{\chi,n}^{-2s}\right)^{\frac{1}{2}}.
\]
Here $\left\langle \cdot,\cdot\right\rangle_{-s}$ denotes the duality bracket 
between $D(A_{\chi}^{-s})$ and $D(A_{\chi}^s)$. 
By the duality, we see that $D(A_{\chi}^{-\frac{1}{2}})$
is embedded continuously into $H^{-1}(\Omega)$, because
$H^1_0(\Omega)$ is embedded continuously into $D(A_{\chi}^{\frac{1}{2}})$.
\\

For any $k\in\mathbb N$, we consider the following condition $(H_k)$:
\begin{equation}\tag{$H_k$}
\rho,\ a_{i,j} \in \mathcal C^{2(k-1)+1}(\overline{\Omega}),\  
i,j = 1,\ldots,d,\quad q\in W^{2(k-1),\infty}(\Omega),\quad \partial\Omega\ 
\textrm{is of $\mathcal C^{2k}$}.
\end{equation}
In view of \cite[Theorem 2.5.1.1]{Gr} (see also 
\cite[Theorem 8.13]{GT}), for any $k\in\mathbb N$,  condition ($H_k$) implies that
the space $D(A_{\chi}^\ell)$ is embedded continuously into $H^{2\ell}(\Omega)$
for any $\ell=0,\ldots,k$ and $\chi=0,1$. 
Therefore, by the interpolation, we deduce that condition ($H_k$) implies that
the space $D(A_{\chi}^s)$ is embedded continuously into $H^{2s}(\Omega)$
for any $s\in[0,k]$.
\\

Let us fix $\theta\in\left[\frac{1}{2},+\infty\right)$, 
$\kappa= \frac{\theta}{2}-\frac{1}{4}$ and $k=1+\lceil\kappa\rceil$. 
Using the above properties and assuming that ($H_k$) is fulfilled, 
for $\mu\geq0$, $h\in H^{-\theta-\chi}(\partial\Omega)$ and 
$\Phi\in D(A_{\chi}^{-\kappa})$, we define the solution 
$y\in D(A_{\chi}^{-\kappa})$ to the following boundary value problem:
\bel{eq2}
\left\{ \begin{array}{rcll} 
\rho(x)^{-1} \cA y(x) +\mu y(x)& = & \Phi(x), & x\in   \Omega,\\
\tau_{\chi} y(x) & = & h(x), & x \in  \pd \Omega, 
\end{array}
\right.
\ee
in the transposition sense:
\bel{trans}
\left\langle y,G\right\rangle_{-\kappa}
= -(-1)^{\chi}\left\langle h,\tau_{\chi}^*(A_{\chi}+\mu)^{-1}G\right\rangle
_{H^{-\theta-\chi}(\partial\Omega),H^{\theta+\chi}(\partial\Omega)}
+ \left\langle \Phi, (A_{\chi}+\mu)^{-1}G\right\rangle_{-1-\kappa}
\ee
for all $G \in D(A_{\chi}^{\kappa})$.
Here $\tau_{\chi}^*$ denotes the formal adjoint operator to $\tau_{\chi}$,
and we have $\tau_0^*=\tau_1$ and $\tau_1^*=\tau_0$. 
Then, we define the solution to \eqref{eq1} in the following way.

\begin{defn}
\label{d1} 
Fix $\theta\in\left[\frac{1}{2},+\infty\right)$, $\kappa= \frac{2\theta-1}{4}$, $k=1+\lceil\kappa\rceil$ and assume that condition $(H_k)$ is fulfilled. 
Let $f\in L^1(0,T; H^{-\theta-\chi}(\pd\Omega))$, 
$F\in L^1(0,T;\rho D(A_{\chi}^{-\kappa -1}))$,  $u_0,u_{\lceil\alpha\rceil-1}\in  D(A_{\chi}^{-\kappa -1})$. We say that $u$ is a weak solution to
\eqref{eq1} if there exist $\epsilon>0$ and
$v\in L^1_{loc}(0,+\infty;D(A_{\chi}^{-\epsilon-\kappa}))$ satisfying  
$u=v_{| Q}$ and the following properties:\\
(i) $\inf\{\lambda>0:\ t\mapsto e^{-\lambda t}v(t,\cdot)\in 
L^1(0,+\infty;D(A_{\chi}^{-\epsilon-\kappa}))\}=0$,\\
(ii) for all $p>0$, the Laplace transform 
$$
\cL v(p,\cdot):=\int_0^{+\infty}e^{-pt}v(t,\cdot)dt
$$ 
of $v$ is lying in $D(A_{\chi}^{-\kappa})$ and it solves the boundary value problem
\bel{eq3}
\left\{ \begin{array}{rcll} 
\rho(x)^{-1} \cA \cL v(p,x) +p^\alpha \cL v(p,x)& = & \int_0^T e^{-pt}\rho^{-1}F(t,x)dt+ \sum_{m=0}^{\lceil\alpha\rceil-1}p^{\alpha-1-m}u_m(x), 
& x\in   \Omega,\\
\tau_{\chi} \cL v(p,x)(x) & = & \int_0^{T}e^{-pt}f(t,x)dt, & x \in  \pd \Omega. 
\end{array}
\right.
\ee

\end{defn}

We give also the following definition of strong solutions to \eqref{eq1}.

\begin{defn}
\label{d3} 

We say that \eqref{eq1} admits a strong solution if there exists a weak solution $u$ to \eqref{eq1} lying in $W^{\lceil\alpha\rceil,1}(0,T;H^{-1}(\Omega))\cap L^1(0,T;H^2(\Omega))$ such that 
\bel{id} 
\rho(x) \partial_t^{\alpha}u+\cA u=F
\ee
holds true in $L^1(0,T;L^2(\Omega))$ and 
\bel{id1} 
\left\{ \begin{array}{rcll} 
\tau_{\chi} u(t,x) & = & f(t,x), & (t,x) \in \Sigma , \\  
\pd_t^k u(0,x) & = & u_k(x), & x \in \Omega,\ k=0,\ldots,\lceil\alpha\rceil-1.
\end{array}
\right.
\ee

\end{defn}

\begin{rmk} Let us observe that for any function $v\in W^{\lceil\alpha\rceil,1}_{loc}(0,+\infty;H^{-1}(\Omega))\cap L^1_{loc}(0,+\infty;H^2(\Omega))$ satisfying
\bel{rr1a}
e^{-pt}v(t,\cdot)\in W^{\lceil\alpha\rceil,1}(0,+\infty;H^{-1}(\Omega))\cap L^1(0,+\infty;H^2(\Omega)),\quad p>0,\ee
we have 
$$
\mathcal L(\rho \partial_t^{\alpha}v+\cA v)(p,\cdot)
= \cA \cL v(p,\cdot) + \rho\left( p^\alpha \cL v(p,\cdot)
- \sum_{m=0}^{\lceil\alpha\rceil-1}p^{\alpha-1-m}v(0,\cdot)\right).
$$

Our choice for the definition   of weak solutions of \eqref{eq1} is based on the above  identity. Moreover, from this identity and the property of 
weak solutions stated in Remark \ref{r1} (see below), one can verify that any strong solution in the sense of Definition \ref{d3}, which can be extended to a function $v\in W^{\lceil\alpha\rceil,1}_{loc}(0,+\infty;H^{-1}(\Omega))\cap L^1_{loc}(0,+\infty;H^2(\Omega))$ satisfying \eqref{rr1a}, will be a weak solution in the sense of Definition \ref{d1}.\end{rmk}

\subsection{Well-posedness  for weak solutions}
In this subsection, we state our results of well-posedness of \eqref{eq1} 
when $f\in L^1(0,T; H^{-\theta-\chi}(\pd\Omega))$, 
$u_0,u_{\lceil\alpha\rceil-1}
\in \rho D(A_{\chi}^{-\kappa -1})$ for some $\theta\in\left[\frac{1}{2},+\infty\right)$ and $\kappa= \frac{2\theta-1}{4}$. More precisely, we consider 
the existence of weak solutions in the sense of Definition \ref{d1}. 
Our first main result can be stated as follows.

\begin{thm}
\label{t1} 
Let $\alpha\in(0,1)\cup(0,2)$, $\chi=0,1$, $\theta\in\left[\frac{1}{2},+\infty\right)$, $\kappa= \frac{\theta}{2}-\frac{1}{4}$, $k=1+\lceil\kappa\rceil$ 
and let $(H_k)$ be fulfilled. Let $r\in[1,+\infty)$,   $\beta$ be given by
$$
\beta= \left\{\begin{array}{ll}1&\ \textrm{ if }r<\alpha^{-1}\\
\alpha^{-1}r^{-1} &\ \textrm{ if } r\geq \alpha^{-1}\end{array}\right.
$$
and let $f\in L^r(0,T; H^{-\theta-\chi}(\pd\Omega))$, 
$F\in L^r(0,T;\rho D(A_{\chi}^{-\frac{\theta}{2} -\frac{3}{2}}))$. 
Consider also 
$$
\left\{\begin{array}{ll}
u_0\in  D(A_{\chi}^{\frac{1}{4}-\beta-\frac{\theta}{2}})&\ \textrm{for }
\alpha\in(0,1),\\
u_0\in  D(A_{\chi}^{-\frac{1}{\alpha r}+\frac{1}{4}-\frac{\theta}{2}}),\ u_1\in  D(A_{\chi}^{-\alpha^{-1}(1+r^{-1})+\frac{1}{4}-\frac{\theta}{2}})&\ \textrm{for }\alpha\in(1,2).
\end{array}\right.
$$
Then problem \eqref{eq1} admits a unique  solution $u$ lying in 
$$
\bigcap_{\epsilon>0}L^r(0,T;D(A_{\chi}^{-\epsilon+\frac{1}{4}
-\frac{\theta}{2}})).
$$
Moreover, for any $\epsilon>0$, we have estimates:
\\
{\bf Case $\alpha\in(0,1)$}:
\bel{t1a}
\norm{u}_{L^r(0,T;D(A_{\chi}^{-\epsilon+\frac{1}{4}-\frac{\theta}{2}}))}
\leq C_\epsilon \left(\norm{f}_{L^r(0,T; H^{-\theta-\chi}(\pd\Omega))}
+ \norm{\rho^{-1}F}_{L^r(0,T; D(A_{\chi}^{-\frac{\theta}{2} -\frac{3}{4}}))}
+ \norm{u_0}_{ D(A_{\chi}^{-\beta+\frac{1}{4}-\frac{\theta}{2}})}\right).
\ee
\\
{\bf Case $\alpha\in(1,2)$}:
\bel{t1b}
\begin{aligned}&\norm{u}_{L^r(0,T;D(A_{\chi}^{-\epsilon+\frac{1}{4}
-\frac{\theta}{2}}))}\\
&\leq C_\epsilon \left(\norm{f}_{L^r(0,T; H^{-\theta-\chi}(\pd\Omega))}
+ \norm{\rho^{-1}F}_{L^r(0,T; D(A_{\chi}^{-1-\kappa))}}
+ \norm{u_0}_{D(A_{\chi}^{-\frac{1}{\alpha r}-\kappa})}
+ \norm{u_1}_{ D(A_{\chi}^{-\frac{(1+r^{-1})}{\alpha}-\kappa})}\right).
\end{aligned}
\ee
In both cases, the constant $C_\epsilon$ depends on 
$\epsilon$, $r$, $\rho$, $\alpha$, $\theta$, $\mathcal A$, $\Omega$ and $T$. 
\end{thm}

\begin{rmk}\label{r1} Note that even if the Definition \ref{d1} of weak solutions depends on the final time $T$, the solution that we obtain in Theorem \ref{t1} is independent of $T$. Namely, fix $T_1,T_2\in(0+\infty)$ with $T_1<T_2$ and consider $f\in L^r(0,T_2; H^{-\theta-\chi}(\pd\Omega))$, $\rho^{-1}F\in L^r(0,T_2; D(A_{\chi}^{-\frac{3}{4}-\frac{\theta}{2}}))$. 
Consider the unique weak solutions $u_\ell$, $\ell=1,2$ to \eqref{eq1} 
for $T=T_\ell$, which are given by Theorem \ref{t1}. 
According to the expression of the weak solution given in the proof of Theorem \ref{t1} in terms of Fourier series, we can verify that 
the restriction of $u_2$ to $(0,T_1)\times\Omega$ coincides with $u_1$.
\end{rmk}

In a special case of $r=2$ and zero initial values, we can improve
Theorem \ref{t1}:

\begin{thm}\label{c1} 
Let the condition of Theorem \ref{t1} be fulfilled 
with $r=2$, $\rho=1$ and $u_0=u_{\lceil\alpha\rceil-1}=0$, and
$\theta \ge \frac{1}{2}$. 
Then the unique weak solution $u$ of \eqref{eq1} is lying in $L^2(0,T;D(A_{\chi}^{\frac{1}{4}-\frac{\theta}{2}}))$ with $\partial_t^\alpha u\in L^2(0,T;D(A_{\chi}^{-\frac{3}{4}-\frac{\theta}{2}}))$. Moreover, we have 
\bel{c1a} \norm{u}_{L^2(0,T;D(A_{\chi}^{\frac{1}{4}-\frac{\theta}{2}}))}+\norm{\partial_t^\alpha u}_{L^2(0,T;D(A_{\chi}^{-\frac{3}{4}-\frac{\theta}{2}})}\leq C\left(\norm{f}_{L^2(0,T;H^{-\theta}(\partial\Omega))}+\norm{F}_{L^2(0,T; D(A_{\chi}^{-\frac{3}{4}-\frac{\theta}{2}}))}\right).\ee
\end{thm}

Under the assumption $F\equiv 0$, 
Theorem 4.1 in \cite{Y} established the same conclusion in the case of
$0<\alpha<1$ and arbitrary $\theta>0$.  On the other hand, this theorem 
holds for $\alpha \in (0,1)\cup (1,2)$ and non-zero $F$ but requires 
that $\theta \ge \frac{1}{2}$.

\subsection{Well-posedness for strong solutions} 

In this subsection, we state our results related to the well-posedness of strong solutions of \eqref{eq1}. We treat separately the case $\alpha\in(0,1)$ and $\alpha\in(1,2)$. Let us start with $\alpha\in(0,1)$. For this purpose,  we introduce a compatibility condition on data 
$f\in W^{1,1}(0,T;H^{-\frac{1}{2}-\chi}(\partial\Omega))$, $F\in W^{1,1}(0,T;\rho D(A_{\chi}^{-1}))$, and $u_0\in L^2(\Omega)$ which requires 
that $u_0$ solves the boundary value problem 
in the transposition sense:
\bel{com1}
\left\{ \begin{array}{rcll} 
 \cA u_0(x)& = &  F(0,x), & x\in   \Omega,\\
\tau_{\chi} u_0(x) & = & f(0,x), & x \in  \pd \Omega, \quad \chi=0,1. 
\end{array}
\right.
\ee
We can now state our result for $\alpha\in(0,1)$.

\begin{thm}
\label{t3} 
Let $\alpha\in(0,1)$, $\chi=0,1$, $\rho\in\mathcal C^2(\overline{\Omega})$ and 
$m\in\mathbb N$ be fixed. 
Assume also that condition $(H_{m})$ is fulfilled. 
Let $u_0\in L^2(\Omega)$, $r\in[1,+\infty)$ and
$$
f\in \bigcap_{k=0}^m W^{m-k,r}(0,T; H^{2k-\frac{1}{2}-\chi}(\pd\Omega)),
\quad F\in \bigcap_{k=1}^m W^{m-k,r}(0,T; H^{2(k-1)}(\Omega))
\cap W^{m,r}(0,T;\rho D(A_{\chi}^{-1})).
$$
For $m\geq2$, we assume that
\bel{tt3a}
\partial_t^{k}f(0,\cdot)=0,\quad \partial_t^{k}F(0,\cdot)=0,\quad k=1,\ldots,
m-1.
\ee
If $u_0$ satisfies the compatibility condition \eqref{com1}, then
the unique weak solution $u$ to \eqref{eq1} is a strong solution lying in
$$
\bigcap_{k=1}^m W^{m-k,r}(0,T; H^{2k}(\Omega))\cap 
W^{m,r}(0,T;H^{-\epsilon}(\Omega)),
$$
where $\epsilon>0$ is arbitrary.
Moreover, for any $\epsilon > 0$, we have
\bel{t3a}
\begin{aligned}
&\norm{u}_{W^{m,r}(0,T; H^{-\epsilon}(\Omega))}
+ \sum_{k=1}^m\norm{u}_{W^{m-k,r}(0,T; H^{2k}(\Omega))}\\
&\leq C_\epsilon \sum_{k=1}^m\left(\norm{f}_{W^{m-k,r}(0,T; 
H^{2k-\frac{1}{2}-\chi}(\pd\Omega))}
+\norm{F}_{ W^{m-k,r}(0,T; H^{2(k-1)}(\Omega))}\right)\\
&\ \ \ 
+ C_\epsilon\left(\norm{f}_{W^{m,r}(0,T; H^{-\frac{1}{2}-\chi}(\pd\Omega))}
+ \norm{\rho^{-1}F}_{ W^{m,r}(0,T; D(A_{\chi}^{-1}))}\right),
\end{aligned}
\ee
with $C_\epsilon$ depending on $\epsilon$, $r$, $\rho$, $\alpha$,  
$\mathcal A$, $\Omega$ and $T$.
\end{thm}                                                       

\begin{rmk}\label{r2} 
We recall that the compatibility condition \eqref{com1} and an a priori 
estimate of $u_0$ yield 
$$
\norm{u_0}_{H^k(\Omega)}\leq C\left(\norm{f(0,\cdot)}_{H^{k-\frac{1}{2}-\chi}(\pd\Omega)}+\norm{\rho^{-1}F(0,\cdot)}_{D(A_{\chi}^{\frac{k-2}{2}})}\right),\quad k=0,2.
$$
Thanks to this estimate, in terms of only $f$ and $F$, we can estimate  
the right-hand sides of \eqref{t3a}, and \eqref{t4b}, \eqref{t5b} which are 
stated below.
\end{rmk}

For $\alpha\in (0,1)$, the compatibility condition \eqref{com1} corresponds 
to an optimal condition guaranteeing the existence of smooth solutions 
as described in Theorem \ref{t3}. Similarly the additional condition \eqref{tt3a} is an optimal condition 
guaranteeing the existence of solutions with higher regularity. Indeed we can prove
\begin{prop}\label{p1} 
For $f\in\mathcal C^\infty(\overline{\Sigma})$, 
$F\in\mathcal C^\infty(\overline{Q})$ and 
$u_0\in\mathcal C^\infty(\overline{\Omega})$, if 
the compatibility condition \eqref{com1} is not fulfilled, then 
the solution $u$ of \eqref{eq1} does not belong to 
$W^{1,(1-\alpha)^{-1}}(0,T;D(A_{\chi}^{-\frac{1}{2}}))$. 
Moreover, if there exists $m\geq2$ such that \eqref{tt3a} fails,  then 
$u\notin W^{m,(1-\alpha)^{-1}}(0,T;D(A_{\chi}^{-\frac{1}{2}}))$.
\end{prop}

For $\alpha\in(1,2)$, we consider two different situations. 
In the first case, we state a result with a weaker regularity assumption 
and the compatibility condition \eqref{com1}. 
In the second result, we consider smoother data under more compatibility 
conditions.  

We start with our first result with only the compatibility condition 
\eqref{com1}.

\begin{thm}
\label{t4} 
Let $\alpha\in(1,2)$, $\chi=0,1$ and let $\delta\in(0,1/4)$, $r\in[1,+\infty)$ satisfy 
\bel{t4a} 
 \quad r< \frac{1}{1-\alpha \delta}.
\ee
Assume $u_0\in H^1(\Omega)$, $u_1\in H^{2\delta}(\Omega)$ and
$$
f\in W^{2,r}(0,T; H^{\frac{1}{2}-\chi}(\pd\Omega))\cap 
L^r(0,T; H^{\frac{3}{2}-\chi}(\pd\Omega)),
$$
$$
F\in W^{2,r}(0,T; D(A_{\chi}^{-\frac{1}{2}}))\cap L^r(0,T; L^2(\Omega)).
$$
If $u_0$ satisfies \eqref{com1}, then the unique weak solution $u$ to
\eqref{eq1} is a strong solution lying in
$$  
L^r(0,T; H^2(\Omega))\cap W^{1,r}(0,T; H^1(\Omega))\cap 
W^{2,r}(0,T; L^2(\Omega)).
$$
Moreover, we have
\bel{t4b}\begin{aligned}&\norm{u}_{W^{2,r}(0,T; L^2(\Omega))}
+ \norm{u}_{W^{1,r}(0,T; H^1(\Omega))}+\norm{u}_{L^r(0,T; H^{2}(\Omega))}\\
&\leq C \left(\norm{f}_{W^{2,r}(0,T; H^{\frac{1}{2}-\chi}(\pd\Omega))}
+ \norm{f}_{L^r(0,T; H^{\frac{3}{2}-\chi}(\pd\Omega))}
+ \norm{u_1}_{H^{2\delta}(\Omega)}\right)    
                                     \\
&\ \ \ + C \left(\norm{F}_{W^{2,r}(0,T; D(A_{\chi}^{-\frac{1}{2}}))}
+ \norm{f}_{L^r(0,T; L^2(\Omega))}\right),
\end{aligned}\ee
where the constant $C>0$ depends on  $r$, $\rho$, $\alpha$,  $\mathcal A$, $\Omega$ and $T$.
\end{thm}

Next we discuss a solutions with higher regularity requiring the two compatibility conditions \eqref{com1} and \eqref{com2}: we assume that 
$u_1$ satisfies  
\bel{com2}
\left\{ \begin{array}{rcll} 
 \cA u_1(x)& = & \partial_tF(0,x), & x\in   \Omega,\\
\tau_{\chi} u_1(x) & = & \partial_tf(0,x), & x \in  \pd \Omega. 
\end{array}
\right.
\ee

\begin{thm}
\label{t5} 
Let $\alpha\in(1,2)$, $\chi=0,1$ and $m_1\in\mathbb N$ be fixed. 
Assume also that condition $(H_{m_1+1})$ is fulfilled and  fix $3\leq m
\leq 2m_1+2$.  Let $u_0\in H^2(\Omega)$, $u_1\in H^1(\Omega)$, $r\in[1,+\infty)$and consider
$$
f\in \bigcap_{k=2}^m W^{m-k,r}(0,T; H^{k-\frac{1}{2}-\chi}(\pd\Omega))
\cap W^{m,r}(0,T; H^{\frac{1}{2}-\chi}(\pd\Omega)),
$$
$$
F\in \bigcap_{k=2}^m W^{m-k,r}(0,T; H^{k-2}(\Omega))
\cap W^{m,r}(0,T; D(A_{\chi}^{-\frac{1}{2}})).
$$
For $m\geq4$, we  assume also that
\bel{t5a}
\partial_t^{k}f(0,\cdot)=0,\quad \partial_t^{k}F(0,\cdot)=0,\quad k=3,
\ldots,m-1.
\ee
If $u_0$ and $u_1$ satisfy the compatibility conditions \eqref{com1} 
and \eqref{com2}, then the unique weak solution $u$ of \eqref{eq1} is a 
strong solution lying in
$$
\bigcap_{k=0}^m W^{m-k,r}(0,T; H^{k}(\Omega)).
$$
Moreover,  we have
\bel{t5b}
\begin{aligned}
\sum_{k=0}^m\norm{u}_{W^{m-k,r}(0,T; H^{k}(\Omega))}
\leq& C \left(\sum_{k=2}^m\norm{f}_{W^{m-k,r}(0,T; H^{k-\frac{1}{2}-\chi}
(\pd\Omega))}
+ \norm{F}_{W^{m-k,r}(0,T; H^{k-2}(\Omega))}\right)\\
\ &+C \left(\norm{f}_{W^{m,r}(0,T; H^{\frac{1}{2}-\chi}(\pd\Omega))}
+ \norm{F}_{W^{m,r}(0,T; D(A_{\chi}^{-\frac{1}{2}}))}\right),
\end{aligned}\ee
where the constant $C>0$ depends on  $\rho$, $r$, $\alpha$,  $\mathcal A$, 
$\Omega$ and $T$.
\end{thm}

For $\alpha\in (1,2)$, the compatibility condition \eqref{com1} corresponds 
to an optimal condition guaranteeing the existence of solutions with 
the smoothness of Theorem \ref{t4}, while both conditions \eqref{com1} 
and \eqref{com2} are required for the existence of smooth solution as stated in Theorem \ref{t5}. In the same way, the additional condition \eqref{t5a} is an optimal 
condition guaranteeing existence of solutions with higher regularity. Indeed we can prove
\begin{prop}\label{p2}
For $f\in\mathcal C^\infty(\overline{\Sigma})$ and $u_0, u_1\in\mathcal 
C^\infty(\overline{\Omega})$, if compatibility condition \eqref{com1} 
is not fulfilled, then the solution $u$ to \eqref{eq1} does not belong 
to $W^{2,(2-\alpha)^{-1}}(0,T;D(A_{\chi}^{-\frac{1}{2}}))$.
Moreover if \eqref{com1} is fulfilled but not \eqref{com2}, then
$u\in W^{3,(2-\alpha)^{-1}}(0,T;D(A_{\chi}^{-\frac{1}{2}}))$ does not hold. Finally,  if there exists $m\geq4$ such that \eqref{t5a} is not fulfilled,  then  $u\notin W^{m,(2-\alpha)^{-1}}(0,T;D(A_{\chi}^{-\frac{1}{2}}))$.
\end{prop}

\subsection{Comments about our results}

To the best of our knowledge,  Theorem \ref{t1} is the first result of well-posedness   of \eqref{eq1} with such weak assumptions imposed on 
the boundary value $f$, the initial values 
$u_0, u_{\lceil\alpha\rceil-1}$ and the source term $F$, as long as 
the elliptic part $\mathcal{A}$ is $t$-independent and symmetric. 
Indeed, all other comparable results have been stated with less general 
equations and smoother data (see for instance \cite{Ke1,Ke2,KR}) or 
with zero initial value (see \cite{Y}). 
In addition, we state these well-posedness for solutions 
lying in $L^r$ in time with $r\in[1,+\infty]$, 
while other comparable results are restricted to solutions lying 
in $L^2$ in time. 


Theorems \ref{t3}, \ref{t4}, and \ref{t5} are concerned with 
the well-posedness in the strong sense of \eqref{eq1}. 
Our aim is to obtain an optimal condition guaranteeing the existence of 
smooth solutions to \eqref{eq1} in time and space. 
It is known that there exist even smooth data satisfying a usual 
compatibility conditions, but that the regularity in time of the solution
to \eqref{eq1} can not exceed $W^{1,(\lceil\alpha\rceil-\alpha)^{-1}}$ 
(see Propositions \ref{p1} and \ref{p2} and Section 6 for more details). 
Moreover, most of results concerning the existence of smooth solutions 
to \eqref{eq1} 
have mainly established some H\"older continuity in time of the solutions (e.g. \cite{Gu, Z2}) or the regularity which is weaker than 
$H^\alpha$ in time (e.g. \cite{KuY}). 
In Theorems \ref{t3} and \ref{t5}, we prove the existence of 
solutions lying $W^{m,r}$ in time with arbitrary $m\in\mathbb N$ and
$r\in[1,+\infty]$ by assuming some compatibility conditions 
\eqref{com1} and \eqref{tt3a} (respectively \eqref{com1}, \eqref{com2}, 
and \eqref{t5a}) 
for $\alpha\in(0,1)$ (respectively $\alpha\in(1,2)$). 


We have obtained Theorem \ref{t1} 
by means of the representation of the solutions which explicitly 
involves the boundary value $f$.
In the same way, we derive and prove an optimality of the compatibility 
conditions \eqref{com1}, \eqref{tt3a}, \eqref{com2} and \eqref{t5a} 
by using the Fourier series representation of the solution to \eqref{eq1}. 
The representation of solutions to \eqref{eq1} which we used for 
the proof of Theorem \ref{t1}, is the key ingredient for the well-posedness
of \eqref{eq1}.

Indeed, in order to reach the regularity stated in Theorems \ref{t3}, \ref{t4} 
and \ref{t5}, 
we do not know whether we can use a classical lifting arguments 
with results on the existence of the solutions with 
homogeneous boundary conditions like \cite{KY1,KuY,SY,Z1,Z2}. 
\\ 

\subsection{Outline}

This paper is organized as follows. Section 2 is devoted to the proof
of Theorems \ref{t1} and \ref{c1}. 
In Section 3 we prove Theorem \ref{t3}.
Sections 4 and 5 are devoted respectively to the proof of Theorems \ref{t4}
and \ref{t5}, while Section 6 is devoted to the proofs of Propositions 
\ref{p1} and \ref{p2} emphasizing the optimality of the compatibility conditions \eqref{com1}, \eqref{tt3a}, \eqref{com2} and \eqref{t5a}.

\section{Proof of Theorems \ref{t1} and  \ref{c1}}

We start with a lemma. Then we prove Theorem \ref{t1} for $\alpha\in(0,1)$, and then for the case $\alpha\in(1,2)$.
Finally, we prove Theorem \ref{c1}.

\subsection{Preliminary lemma}

Henceforth $C>0$ denotes generic constants which depend 
on $\mathcal A$, $\rho$, $\theta$ and $r, \alpha, T$ and $\Omega$, and 
$C$ may change from line to line.

We prove

\begin{lem}\label{l1} Let $\chi=0,1$, $\theta\in\left[\frac{1}{2},+\infty\right)$, $\kappa= \frac{2\theta-1}{4}$, $k=1+\lceil\kappa\rceil$ and 
let condition $(H_k)$ be fulfilled. Then, for any $h\in H^{-\theta-\chi}(\partial\Omega)$, we have
\bel{l1a} 
\sum_{n=1}^\infty \abs{\lambda_{\chi,n}^{-1-\kappa}\left\langle h,\tau_{\chi}^*\phi_{\chi,n}\right\rangle_{H^{-\theta-\chi}(\partial\Omega),H^{\theta+\chi}(\partial\Omega)}}^2\leq C^2 \norm{h}_{H^{-\theta-\chi}(\partial\Omega)}^2.
\ee
\end{lem}
\begin{proof}  
Let $y\in D(A_{\chi}^{-\kappa})$ be the solution in the transposition sense 
to \eqref{eq2} with $\mu=0$ and $\Phi=0$. 
Since $A_{\chi}\phi_{\chi,n}=\lambda_{\chi,n}\phi_{\chi,n}$, we have 
$$
\left\langle y,\lambda_{\chi,n}\phi_{\chi,n}\right\rangle_{-\kappa}=-(-1)^\chi\left\langle h,\tau_{\chi}^*\phi_{\chi,n}\right\rangle_{H^{-\theta-\chi}(\partial\Omega), H^{\theta+\chi}(\partial\Omega)}, \quad n\in \mathbb N.
$$
Therefore 
$$
\left\langle y,\phi_{\chi,n}\right\rangle_{-\kappa}=-(-1)^\chi\lambda_{\chi,n}^{-1}\left\langle h,\tau_{\chi}^*\phi_{\chi,n}\right\rangle_{H^{-\theta-\chi}(\partial\Omega),H^{\theta+\chi}(\partial\Omega)}
$$
and 
\bel{l1c}
\sum_{n=1}^\infty \abs{\lambda_{\chi,n}^{-1-\kappa}\left\langle h,\tau_{\chi}^*\phi_{\chi,n}\right\rangle_{H^{-\theta-\chi}(\partial\Omega),H^{\theta+\chi}(\partial\Omega)}}^2=\norm{y}_{D(A_{\chi}^{-\kappa})}^2.\ee
Furthermore, for any $G\in D(A_{\chi}^\kappa)$, we have
$$\begin{aligned}
\abs{\left\langle y,G\right\rangle_{-\kappa}}&=\abs{\left\langle h,\tau_{\chi}^*A_{\chi}^{-1}G\right\rangle_{H^{-\theta-\chi}(\partial\Omega),H^{\theta+\chi}(\partial\Omega)}}\\
\ &\leq \norm{h}_{H^{-\theta-\chi}(\partial\Omega)}\norm{\tau_{\chi}^*A_{\chi}^{-1}G}_{H^{\theta+\chi}(\partial\Omega)}\\
\ &\leq C\norm{h}_{H^{-\theta-\chi}(\partial\Omega)}\norm{A_{\chi}^{-1}G}_{H^{2(1+\kappa)}(\Omega)}\\
\ &\leq C\norm{h}_{H^{-\theta-\chi}(\partial\Omega)}\norm{A_{\chi}^{-1}G}_{D(A_{\chi}^{1+\kappa})}\\
\ &\leq C\norm{h}_{H^{-\theta-\chi}(\partial\Omega)}\norm{G}_{D(A_{\chi}^{\kappa})}.
\end{aligned}$$
Therefore, we obtain
$$
\norm{y}_{D(A_{\chi}^{-\kappa})}\leq C\norm{h}_{H^{-\theta-\chi}
(\partial\Omega)}
$$
and combining this with \eqref{l1c}, we deduce \eqref{l1a}.
\end{proof}

By this lemma, we are now in position to complete the proof of Theorem \ref{t1}.

\subsection{Proof of Theorem \ref{t1} for $\alpha\in(0,1)$}
For all $n\in\mathbb N$, we set
$$
u_{1,n}(t):=-(-1)^\chi\int_0^t(t-s)^{\alpha-1}E_{\alpha,\alpha}(-\lambda_{\chi,n}(t-s)^\alpha)\mathds{1}_{(0,T)}(s)\left\langle f(s,\cdot),\tau_{\chi}^*\phi_{\chi,n}\right\rangle_{H^{-\theta-\chi}(\partial\Omega),H^{\theta+\chi}(\partial\Omega)}ds,
$$
$$
u_{2,n}(t):=E_{\alpha,1}(-\lambda_{\chi,n}t^\alpha)\left\langle u_0,\phi_{\chi,n}\right\rangle_{-1-\kappa},$$
$$ 
u_{3,n}(t):=\int_0^t(t-s)^{\alpha-1}E_{\alpha,\alpha}(-\lambda_{\chi,n}(t-s)^\alpha)\mathds{1}_{(0,T)}(s)\left\langle \rho^{-1}F(s,\cdot),\phi_{\chi,n}\right\rangle_{-\kappa-1}ds,
$$
and
$$
u_n(t):=u_{1,n}(t)+u_{2,n}(t)+u_{3,n}(t),
$$
where $\mathds{1}_{(0,T)}$ denotes the characteristic function of $(0,T)$. 
Here, for $\beta_1,\beta_2>0$, $E_{\beta_1,\beta_2}$ denotes  
the Mittag-Leffler function given by  
$$
E_{\beta_1,\beta_2}(z)=\sum_{k=0}^\infty \frac{z^k}{\Gamma(\beta_1 k+\beta_2)},
\quad   z\in\mathbb C.
$$
\\

We will divide the proof of Theorem \ref{t1} into three steps:
In the first step, we will prove that for any $\epsilon>0$, the sequence 
$$
\sum_{n=1}^N u_n(t)\phi_{\chi,n}, \quad N \in \mathbb N,
$$
converges to $u$ in $L^r(0,T;D(A_{\chi}^{-\epsilon-\kappa}))$ as $N\to \infty$.
In the second step, we prove that the same sequence converges 
in $L^r_{loc}(0,+\infty;D(A_{\chi}^{-\epsilon-\kappa}))$ to a function $v$ satisfying the conditions (i) and (ii) of Definition \ref{d1}. 
In the third step, we complete the proof of the theorem. 

In addition to the generic constant $C>0$, for all $\epsilon>0$
by $C_\epsilon >0$ we denote generic constants depending also on 
$\epsilon$.

\textbf{Step 1.} Fix $\epsilon>0$. Let us show that the sequence
$$
\sum_{n=1}^N u_n(t)\phi_{\chi,n}, \quad N\in \mathbb N
$$
converges in $L^r(0,T;D(A_{\chi}^{-\epsilon-\kappa}))$. 
For this purpose, it suffices to prove that the sequences
$$
\sum_{n=1}^Nu_{i,n}(t)\phi_{\chi,n}, \quad i=1,2,3, \, N \in \mathbb N,
$$
converge in $L^r(0,T;D(A_{\chi}^{-\epsilon-\kappa}))$. 
First we prove this result for the case $i=1$. 
For $m<n$ and almost all $t\in(0,T)$, we have
$$\begin{aligned}
&\norm{\sum_{\ell=m}^{n}u_{1,\ell}(t)\phi_{\chi,\ell}}
_{D(A_{\chi}^{-\epsilon-\kappa})}
                           \\
&= \norm{\sum_{\ell=m}^{n}\left(\int_0^t(t-s)^{\alpha-1}E_{\alpha,\alpha}
(-\lambda_{\chi,\ell}(t-s)^\alpha)\left\langle f(s,\cdot),\tau_{\chi}^*\phi
_{\chi,\ell}\right\rangle_{H^{-\theta-\chi}(\partial\Omega),H^{\theta+\chi}(\partial\Omega)}ds\right)\phi_{\chi,\ell}}_{D(A_{\chi}^{-\epsilon-\kappa})}\\
&\leq\int_0^t\norm{\sum_{\ell=m}^{n}(t-s)^{\alpha-1}E_{\alpha,\alpha}
(-\lambda_{\chi,\ell}(t-s)^\alpha)\left\langle f(s,\cdot),\tau_{\chi}^*
\phi_{\chi,\ell}\right\rangle_{H^{-\theta-\chi}(\partial\Omega),
H^{\theta+\chi}(\partial\Omega)}\phi_{\chi,\ell}}
_{D(A_{\chi}^{-\epsilon-\kappa})}ds.
\end{aligned}$$
In view of formula (1.148) of \cite[Theorem 1.6]{P}, for almost all $s\in(0,t)$,
we find
$$\begin{aligned} 
&\norm{\sum_{\ell=m}^{n}(t-s)^{\alpha-1}E_{\alpha,\alpha}(-\lambda_{\chi,\ell}
(t-s)^\alpha)\left\langle f(s,\cdot),\tau_{\chi}^*\phi_{\chi,\ell}
\right\rangle_{H^{-\theta-\chi}(\partial\Omega),H^{\theta+\chi}(\partial\Omega)}\phi_{\chi,\ell}}_{D(A_{\chi}^{-\epsilon-\kappa})}^2\\
&\leq \sum_{\ell=m}^{n}\abs{\lambda_{\chi,\ell}^{-\epsilon-\kappa}
(t-s)^{\alpha-1}E_{\alpha,\alpha}(-\lambda_{\chi,\ell}(t-s)^\alpha)
\left\langle f(s,\cdot),\tau_{\chi}^*\phi_{\chi,\ell}\right\rangle_{H^{-\theta-\chi}(\partial\Omega),
H^{\theta+\chi}(\partial\Omega)}}^2\\
&\leq C (t-s)^{2(\epsilon\alpha-1)}\sum_{\ell=m}^{n}\abs{\lambda_{\chi,\ell}
^{-1-\kappa}\left\langle f(s,\cdot),\tau_{\chi}^*\phi_{\chi,\ell}
\right\rangle_{H^{-\theta-\chi}(\partial\Omega),H^{\theta+\chi}
(\partial\Omega)}}^2.
\end{aligned}$$
Hence 
$$
\norm{\sum_{\ell=m}^{n}u_{1,\ell}(t)\phi_{\chi,\ell}}_{D(A_{\chi}
^{-\epsilon-\kappa})}
\leq C\int_0^t(t-s)^{\epsilon\alpha-1}\left(\sum_{\ell=m}^{n}
\abs{\lambda_{\chi,\ell}^{-1-\kappa}\left\langle f(s,\cdot),\tau_{\chi}^*
\phi_{\chi,\ell}\right\rangle_{H^{-\theta-\chi}(\partial\Omega),H^{\theta+\chi}(\partial\Omega)}}^2\right)
^{\frac{1}{2}}ds.
$$
Applying the young inequality for convolution, we obtain
$$\begin{aligned} 
&\norm{\sum_{\ell=m}^{n}u_{1,\ell}(t)\phi_{\chi,\ell}}_{L^r(0,T;D(A_{\chi}^{-\epsilon-\kappa}))}\\
&\leq C(\epsilon\alpha)^{-1}T^{\epsilon\alpha}
\norm{\sum_{\ell=1}^{n}\lambda_{\chi,\ell}^{-1-\kappa}
\left\langle f(t,\cdot),\tau_{\chi}^*\phi_{\chi,\ell}\right\rangle
\phi_{\chi,\ell}
- \sum_{\ell=1}^{m}\lambda_{\chi,\ell}^{-1-\kappa}\left\langle f(t,\cdot),\tau_{\chi}^*\phi_{\chi,\ell}\right\rangle\phi_{\chi,\ell}}
_{L^r(0,T;L^2(\Omega))}.
\end{aligned}$$
Using Lemma \ref{l1}, for almost all $t\in(0,T)$ and all 
$N\in\mathbb N$, we obtain 
$$\begin{aligned} 
&\norm{\sum_{\ell=1}^N\lambda_{\chi,\ell}^{-1-\kappa}\left\langle f(t,\cdot),
\tau_{\chi}^*\phi_{\chi,\ell}\right\rangle\phi_{\chi,\ell}}_{L^2(\Omega)}\\
&\leq\left(\sum_{\ell=1}^{+\infty}\abs{\lambda_{\chi,\ell}^{-1-\kappa}
\left\langle f(t,\cdot),\tau_{\chi}^*\phi_{\chi,\ell}
\right\rangle_{H^{-\theta-\chi}(\partial\Omega),H^{\theta+\chi}(\partial\Omega)}}^2\right)^{\frac{1}{2}}\leq C\norm{f(t,\cdot)}_{H^{-\theta-\chi}
(\partial\Omega)}
\end{aligned}$$
and the limit
$$
\lim_{N \to \infty}
\sum_{\ell=1}^N \lambda_{\chi,\ell}^{-1-\kappa}\left\langle f(t,\cdot),
\tau_{\chi}^*\phi_{\chi,\ell}\right\rangle_{H^{-\theta-\chi}(\partial\Omega),H^{\theta+\chi}(\partial\Omega)}\phi_{\chi,\ell},
$$
exists in $L^2(\Omega)$ for almost all $t\in(0,T)$. 
In the  same way, by $f\in L^r(0,T;H^{-\theta-\chi}(\partial\Omega))$, 
applying the Lebesgue dominate convergence theorem for functions taking values in $L^2(\Omega)$, we deduce that the limit
$$
\lim_{N \to \infty} 
\sum_{\ell=1}^N \lambda_{\chi,\ell}^{-1-\kappa}
\left\langle f(t,\cdot),\tau_{\chi}^*\phi_{\chi,\ell}\right\rangle
_{H^{-\theta-\chi}(\partial\Omega),H^{\theta+\chi}(\partial\Omega)}
\phi_{\chi,\ell},
$$
exists in $L^r(0,T;L^2(\Omega))$.
In particular, it is a Cauchy sequence and so we have
$$
\lim_{m,n\to\infty}\norm{\sum_{\ell=1}^{n}\lambda_{\chi,\ell}^{-1-\kappa}
\left\langle f(t,\cdot),\tau_{\chi}^*\phi_{\chi,\ell}\right\rangle
\phi_{\chi,\ell}-\sum_{\ell=1}^{m}\lambda_{\chi,\ell}^{-1-\kappa}
\left\langle f(t,\cdot),\tau_{\chi}^*\phi_{\chi,\ell}\right\rangle
\phi_{\chi,\ell}}_{L^r(0,T;L^2(\Omega))}=0,
$$
which implies that
$$
\lim_{m,n\to\infty}\norm{\sum_{\ell=m}^{n}u_{1,\ell}(t)\phi_{\chi,\ell}}
_{L^r(0,T;D(A_{\chi}^{-\epsilon-\kappa}))}=0.
$$
Thus the sequence
$$
\sum_{n=1}^Nu_{1,n}(t)\phi_{\chi,n},\quad N\in\mathbb N
$$
is a Cauchy sequence in $L^r(0,T;D(A_{\chi}^{-\epsilon-\kappa}))$,
which yields the convergence in $L^r(0,T;D(A_{\chi}^{-\epsilon-\kappa}))$.
Similarly we can show that the sequence
$$
\sum_{n=1}^Nu_{3,n}(t)\phi_{\chi,n}, \quad N\in \mathbb N,
$$
converges in $L^r(0,T;D(A_{\chi}^{-\epsilon-\kappa}))$. 
Moreover formula (1.148) of \cite[Theorem 1.6]{P} yields 
$$
\sum_{n=1}^{\infty}\abs{\lambda_{\chi,n}^{-\epsilon-\kappa}u_{2,n}(t)}^2
\leq C t^{2(\epsilon-\beta)\alpha}\sum_{n=1}^{\infty}\abs{\lambda_{\chi,n}^{-\beta-\kappa}\left\langle u_0,\phi_{\chi,n}\right\rangle}^2\leq Ct^{2(\epsilon-\beta)\alpha}\norm{u_0}_{D(A_{\chi}^{-\kappa-\beta})}^2
$$
and, combining this with the fact that $r\beta \alpha=1$, we can conclude that
the limit 
$$
\lim_{N\to\infty} \sum_{n=1}^Nu_{2,n}(t)\phi_{\chi,n}, 
$$
exists in $L^r(0,T;D(A_{\chi}^{-\epsilon-\kappa}))$. These results prove 
that the sequence
$$
\sum_{n=1}^N u_n(t)\phi_{\chi,n}, \quad N\in \mathbb N,
$$
converges in $L^r(0,T;D(A_{\chi}^{-\epsilon-\kappa}))$.

\textbf{Step 2.} Fix $\epsilon>0$. Using the arguments of 
Step 1, we can define $v\in L^r_{loc}(0,+\infty;D(A_{\chi}^{-\epsilon-\kappa}))$ by
$$
v(t,\cdot):=\sum_{n=1}^\infty u_n(t)\phi_{\chi,n}.
$$
In this step, we will show that $v$ fulfills conditions (i) and (ii) of 
Definition \ref{d1}. Fixing $p>0$ and repeating the arguments of Step 1, we deduce that, for almost all $t>0$, we have 
$$\begin{aligned}
&\norm{e^{-pt}\sum_{n=1}^\infty u_{1,n}(t)\phi_{\chi,n}}_{D(A_{\chi}^{-\epsilon-\kappa})}\\
&\leq C \int_0^te^{-p(t-s)}(t-s)^{\epsilon\alpha-1}e^{-ps}\mathds{1}_{(0,T)}(s)\left(\sum_{n=1}^{+\infty}\abs{\lambda_{\chi,n}^{-1-\kappa}\left\langle f(s,\cdot),\tau_{\chi}^*\phi_{\chi,n}\right\rangle_{H^{-\theta-\chi}(\partial\Omega),H^{\theta+\chi}(\partial\Omega)}}^2\right)^{\frac{1}{2}}ds.
\end{aligned}$$
Combining this with \eqref{l1a} and applying Young's inequality, 
we obtain
$$\begin{aligned}
&\norm{e^{-pt}\sum_{n=1}^\infty u_{1,n}(t)\phi_{\chi,n}}_{L^1(0,+\infty;D(A_{\chi}^{-\epsilon-\kappa}))}\\
&\leq C\norm{\left(e^{-pt}t^{\epsilon\alpha-1}\mathds{1}_{(0,+\infty)}(t)\right)*\left(\mathds{1}_{(0,T)}(t)\norm{f(t,\cdot)}_{H^{-\theta-\chi}(\partial\Omega)}\right)}_{L^1(0,+\infty)}\\
&\leq C\left(\int_0^{+\infty}e^{-pt}t^{\epsilon\alpha-1}dt\right)\norm{f}_{L^r(0,T;H^{-\theta-\chi}(\partial\Omega))}<\infty.
\end{aligned}$$
Therefore, we have
$$
e^{-pt}\sum_{n=1}^\infty u_{1,n}(t)\phi_{\chi,n}\in L^1(0,+\infty;D(A_{\chi}^{-\epsilon-\kappa})),\quad p>0.
$$
In the same way, one can verify that 
$$
e^{-pt}\sum_{n=1}^\infty u_{i,n}(t)\phi_{\chi,n}\in L^1(0,+\infty;D(A_{\chi}^{-\epsilon-\kappa})),\quad p>0,\ i=2,3
$$
and conclude that condition (i) of Definition \ref{d1} is fulfilled. 

Now let us consider condition (ii) of Definition \ref{d1}. Applying the properties of Mittag-Leffter functions (see for instance \cite{P}), we see
that the Laplace transform $\mathcal Lv(p,\cdot)$ of $v$ satisfies 
$$
\mathcal Lv(p,\cdot) = \sum_{n=1}^\infty \mathcal{L}u_n(p)\phi_{\chi,n}
\in D(A_{\chi}^{-\epsilon-\kappa}),
$$
where  
$$\begin{aligned}
&\mathcal Lu_n(p)=\int_0^{+\infty}e^{-pt}u_n(t)dt\\
&=\frac{\left\langle  p^{\alpha-1}u_0+\int_0^T e^{-pt}\rho^{-1}F(t,\cdot)dt,\phi_{\chi,n}\right\rangle_{-\kappa-1}}{p^\alpha+\lambda_{\chi,n}}-\frac{(-1)^\chi\left\langle \int_0^Te^{-pt}f(t,\cdot)dt,\tau_{\chi}^*\phi_{\chi,n}\right\rangle_{H^{-\theta-\chi},H^{\theta+\chi}}}{p^\alpha+\lambda_{\chi,n}}.
\end{aligned}$$
Now let $V(p)$ be the solution in the transposition sense of the boundary 
value problem
$$\left\{ \begin{array}{rcll} 
\rho(x)^{-1} \cA V(p,x) +p^\alpha V(p,x)& = & \int_0^T e^{-pt}\rho^{-1}F(t,x)dt+p^{\alpha-1}u_0(x), 
& x\in   \Omega,\\
\tau_{\chi} V(p,x)(x) & = & \int_0^{T}e^{-pt}f(t,x)dt, & x \in  \pd \Omega. 
\end{array}
\right.$$
By the definition of the transposition, setting 
$G = \phi_{\chi,n}$ in \eqref{trans} and applying 
$(\mathcal A+p^\alpha)\phi_{\chi,n}=(\lambda_{\chi,n}+p^\alpha)\phi_{\chi,n}$, 
we obtain
$$\begin{aligned}
&(\lambda_{\chi,n}+p^\alpha)\left\langle V(p,\cdot),\phi_{\chi,n}
\right\rangle_{-\kappa}\\
&=-(-1)^\chi\left\langle \int_0^Te^{-pt}f(t,\cdot)dt,\tau_{\chi}^*\phi_{\chi,n}
\right\rangle_{H^{-\theta-\chi}(\partial\Omega),
H^{\theta+\chi}(\partial\Omega)}\\
& + \left\langle  p^{\alpha-1}u_0+\int_0^T e^{-pt}\rho^{-1}
F(t,\cdot)dt,\phi_{\chi,n}\right\rangle_{-\kappa-1}\\
&=(\lambda_{\chi,n}+p^\alpha)\mathcal{L}u_n(p).
\end{aligned}$$
Therefore $\mathcal{L}v(p,x) = V(p,x)$ for $p>0$ and $x \in \Omega$, and so
we can verify condition (ii) of Definition \ref{d1}. 

\textbf{Step 3.} In this step we complete the proof of Theorem \ref{t1}. 
The above argument shows that $u$ is a weak solution to \eqref{eq1} in the sense of Definition \ref{d1}. Moreover, the uniqueness of this solution is guaranteed by the uniqueness of the Laplace transform of a function and the uniqueness of the solutions of the boundary value problem \eqref{eq3}. Therefore, the proof of the theorem will be completed if we show that for all $\epsilon>0$, 
estimate \eqref{t1a} holds true. For this purpose, let use first consider
$$
\norm{\sum_{n=1}^{\infty}u_{1,n}\phi_{\chi,n}}
_{L^r(0,T;D(A_{\chi}^{-\epsilon-\kappa}))}.
$$
Repeating the arguments of Step 1 and Lemma \ref{l1}, we deduce that
$$
\norm{\sum_{n=1}^{\infty}u_{1,n}\phi_{\chi,n}}_{L^r(0,T;D(A_{\chi}^{-\epsilon-\kappa}))}\leq C\norm{\int_0^t(t-s)^{\epsilon\alpha-1}\norm{f(s,\cdot)}_{H^{-\theta-\chi}(\partial\Omega)}ds}_{L^r(0,T)}.
$$
Therefore Young's inequality for convolution yields
$$
\norm{\sum_{n=1}^{\infty}u_{1,n}\phi_{\chi,n}}_{L^r(0,T;D(A_{\chi}^{-\epsilon-\kappa}))}\leq C_\epsilon \norm{f}
_{L^r(0,T;H^{-\theta-\chi}(\partial\Omega))}.$$
In the same way, we obtain 
$$
\norm{\sum_{n=1}^{\infty}u_{i,n}\phi_{\chi,n}}_{L^r(0,T;D(A_{\chi}
^{-\epsilon-\kappa}))}\leq C_\epsilon\left(\norm{u_0}_{D(A_{\chi}^{-\kappa-\beta})}+\norm{\rho^{-1}F}_{L^r(0,T;D(A_{\chi}^{-\kappa-1}))}\right),\ i=2,3.
$$
Combining these estimates, we deduce \eqref{t1a}. This completes the proof of 
Theorem \ref{t1} for $\alpha\in(0,1)$.

\subsection{Proof of Theorem \ref{t1} for $\alpha\in(1,2)$}
For this purpose we fix $\epsilon\in(0,\alpha^{-1})$.
For all $n\in\mathbb N$, we set
$$
u_{1,n}(t):=-(-1)^\chi\int_0^t(t-s)^{\alpha-1}E_{\alpha,\alpha}(-\lambda_{\chi,n}(t-s)^\alpha)\mathds{1}_{(0,T)}(s)\left\langle f(s,\cdot),\tau_{\chi}^*\phi_{\chi,n}\right\rangle_{H^{-\theta-\chi}(\partial\Omega),H^{\theta+\chi}(\partial\Omega)}ds,
$$
$$
u_{2,n}(t):=E_{\alpha,1}(-\lambda_{\chi,n}t^\alpha)\left\langle u_0,\phi_{\chi,n}\right\rangle_{-\alpha^{-1}-\kappa},
$$
$$ 
u_{3,n}(t):=tE_{\alpha,2}(-\lambda_{\chi,n}t^\alpha)\left\langle  u_1,\phi_{\chi,n}\right\rangle_{-\alpha^{-1}(1+r^{-1})-\kappa},
$$
$$ 
u_{4,n}(t):=\int_0^t(t-s)^{\alpha-1}E_{\alpha,\alpha}(-\lambda_{\chi,n}(t-s)^\alpha)\mathds{1}_{(0,T)}(s)\left\langle \rho^{-1}F(s,\cdot),\phi_{\chi,n}\right\rangle_{-\kappa-1}ds
$$
and
$$
u_n := u_{1,n}+u_{2,n}+u_{3,n}+u_{4,n}.
$$
We will show that the sequence 
$$
\sum_{n=1}^Nu_n(t)\phi_{\chi,n},\quad N\in \mathbb N,
$$
converges in $L^r(0,T;D(A_{\chi}^{-\epsilon-\kappa}))$ to the unique solution 
$u$ to \eqref{eq1}. For this purpose, we remark that formula (1.148) of \cite[Theorem 1.6]{P} implies that
$$
|u_{2,n}(t)|\leq Ct^{\epsilon\alpha-r^{-1}}\lambda_{\chi,n}^{\epsilon-\alpha^{-1}r^{-1}}\abs{\left\langle  u_0,\phi_{\chi,n}\right\rangle
_{-\alpha^{-1}-\kappa}}
$$
and
$$
|u_{3,n}(t)|\leq Ct^{\epsilon\alpha -r^{-1}}\lambda_{\chi,n}^{\epsilon-\alpha^{-1}(1-r^{-1})}\abs{\left\langle  u_1,\phi_{\chi,n}\right\rangle_{-\alpha^{-1}(1+r^{-1})-\kappa}}.
$$
Therefore, repeating the arguments for the case $\alpha\in(0,1)$,
we can complete the proof of Theorem \ref{t1} for $\alpha\in(1,2)$.

\subsection{Proof of Theorem \ref{c1}}
Since $u_0=u_{\lceil\alpha\rceil-1}=0$ and $\rho=1$, the unique weak solution 
$u$ to \eqref{eq1} is given by
$$
u(t,\cdot)=\sum_{n=1}^\infty u_n(t)\phi_{\chi,n},\quad t\in(0,T),
$$
where
$$
u_n(t):=\int_0^t(t-s)^{\alpha-1}E_{\alpha,\alpha}(-\lambda_{\chi,n}(t-s)^\alpha)\lambda_{\chi,n}^{\kappa+1} G_n(s)ds,
$$
and
$$
G_n(t):=\frac{(-1)^{\chi+1}\left\langle f(t,\cdot),\tau_{\chi}^*\phi_{\chi,n}\right\rangle_{H^{-\theta-\chi}(\partial\Omega),H^{\theta+\chi}(\partial\Omega)}+\left\langle F(t,\cdot),\phi_{\chi,n}\right\rangle_{-\kappa-1}}{\lambda_{\chi,n}^{\kappa+1}},\quad n\in\mathbb N,\ t\in(0,T).
$$
Here we notice that
$$
\left\langle A_{\chi}^{-1-\kappa}u(t,\cdot),\phi_{\chi,n} \right\rangle=\lambda_{\chi,n}^{-\kappa-1}u_n(t)=\int_0^t(t-s)^{\alpha-1}E_{\alpha,\alpha}(-\lambda_{\chi,n}(t-s)^\alpha)G_n(s)ds.
$$
Applying Lemma \ref{l1}, we see 
$$
G(t,\cdot)=\sum_{n=1}^\infty G_n(t)\phi_{\chi,n}\in L^2(\QQQ),
$$
and
\bel{c1b}
\norm{G}_{L^2(\QQQ)}\leq C\left(\norm{f}_{L^2(0,T;H^{-\theta}(\partial\Omega))}+\norm{F}_{L^2(0,T; D(A_{\chi}^{-\kappa -1}))}\right).
\ee
On the other hand, in view of \cite[Lemma A.2]{KY1} (see also \cite[Theorem 2.2]{SY}), we have $A_{\chi}^{-1-\kappa}u\in L^2(0,T;D(A_{\chi}))$,
$A_{\chi}^{-1-\kappa}\partial_t^\alpha u=\partial_t^\alpha A_{\chi}^{-1-\kappa}u\in L^2(\QQQ)$ and
\bel{c1c}
\norm{A_{\chi}^{-1-\kappa}u}_{L^2(0,T;D(A_{\chi}))}
+ \norm{\partial_t^\alpha A_{\chi}^{-1-\kappa}u}_{L^2(\QQQ)}
\leq C\norm{G}_{L^2(\QQQ)}.
\ee
From these results, one can easily verify that $u\in L^2(0,T;D(A_{\chi}^{-\kappa}))$, $\partial_t^\alpha u\in L^2(0,T;D(A_{\chi}^{-1-\kappa}))$ and 
estimates \eqref{c1b}-\eqref{c1c} imply \eqref{c1a}. This completes 
the proof of the theorem.

\section{Proof of Theorem \ref{t3}}
\subsection{For $m=1$}
Let us start with the case $m=1$. For this purpose, 
we fix $\epsilon\in(0,1/4)$.
In view of Theorem \ref{t1}, the solution $u\in L^r(0,T;D(A_{\chi}
^{-\epsilon})$ to \eqref{eq1} is given by
$$
u(t,\cdot)=\sum_{n=1}^\infty u_n(t)\phi_{\chi,n},
$$
where we set
$$
u_{1,n}(t):=-(-1)^\chi\int_0^t(t-s)^{\alpha-1}E_{\alpha,\alpha}(-\lambda_{\chi,n}(t-s)^\alpha)\left\langle f(s,\cdot),\tau_{\chi}^*\phi_{\chi,n}\right\rangle_{L^2(\partial\Omega)}ds,
$$
$$
u_{2,n}(t):=E_{\alpha,1}(-\lambda_{\chi,n}t^\alpha)\left\langle  u_0,\phi_{\chi,n}\right\rangle,
$$
$$ 
u_{3,n}(t)=\int_0^t(t-s)^{\alpha-1}E_{\alpha,\alpha}(-\lambda_{\chi,n}(t-s)
^\alpha)\mathds{1}_{(0,T)}(s)\left\langle \rho^{-1}F(s,\cdot),\phi_{\chi,n}
\right\rangle ds
$$
and
$$ 
u_n(t)=u_{1,n}(t)+u_{2,n}(t)+u_{3,n}(t), \quad n\in\mathbb N.
$$
First we see that $u_{1,n}\in W^{1,1}(0,T)$ and
\bel{t3c}\begin{aligned}
&\frac{du_{1,n}}{dt}(t) =: u_{1,n}'(t)
=-(-1)^\chi\partial_t\left(\int_0^ts^{\alpha-1}E_{\alpha,\alpha}(-\lambda
_{\chi,n}s^\alpha)\left\langle f(t-s,\cdot),\tau_{\chi}^*\phi_{\chi,n}\right\rangle_{L^2(\partial\Omega)}ds\right)\\
= & -(-1)^\chi\left\langle f(0,\cdot),\tau_{\chi}^*\phi_{\chi,n}\right\rangle_{H^{-\frac{1}{2}-\chi}(\partial\Omega),H^{\frac{1}{2}+\chi}(\partial\Omega)}t^{\alpha-1}E_{\alpha,\alpha}(-\lambda_{\chi,n}t^\alpha)\\
- & (-1)^\chi\int_0^ts^{\alpha-1}E_{\alpha,\alpha}(-\lambda_{\chi,n}s^\alpha)\left\langle \partial_tf(t-s,\cdot),\tau_{\chi}^*\phi_{\chi,n}\right\rangle_{H^{-\frac{1}{2}-\chi}(\partial\Omega),H^{\frac{1}{2}+\chi}(\partial\Omega)}ds.
\end{aligned}\ee
Similarly, using \cite[Lemma 3.2]{SY}, we deduce that 
$u_{i,n}\in W^{1,1}(0,T)$, $i=2,3$, where 
$$
u_{2,n}'(t)=-\lambda_{\chi,n}\left\langle  u_0,\phi_{\chi,n}\right\rangle t^{\alpha-1}E_{\alpha,\alpha}(-\lambda_{\chi,n}t^\alpha),
$$
$$\begin{aligned}
u_{3,n}'(t)=&\left\langle \rho^{-1} F(0,\cdot),\phi_{\chi,n}\right\rangle_{-1}t^{\alpha-1}E_{\alpha,\alpha}(-\lambda_{\chi,n}t^\alpha)\\
\ &+\int_0^ts^{\alpha-1}E_{\alpha,\alpha}(-\lambda_{\chi,n}s^\alpha)\left\langle \rho^{-1} \partial_tF(t-s,\cdot),\phi_{\chi,n}\right\rangle_{-1}ds.
\end{aligned}$$
On the other hand, the compatibility condition \eqref{com1} and the 
representation \eqref{trans} of the solution in the transposition sense to the elliptic problem \eqref{eq2},
imply that
$$
\lambda_{\chi,n}\left\langle  u_0,\phi_{\chi,n}\right\rangle=-(-1)^\chi\left\langle f(0,\cdot),\tau_{\chi}^*\phi_{\chi,n}\right\rangle_{H^{-\frac{1}{2}-\chi}(\partial\Omega),H^{\frac{1}{2}+\chi}(\partial\Omega)}+\left\langle \rho^{-1} F(0,\cdot),\phi_{\chi,n}\right\rangle_{-1}.
$$
Therefore, we have
$$
u_{2,n}'(t)=\left((-1)^\chi\left\langle f(0,\cdot),\tau_{\chi}^*\phi_{\chi,n}\right\rangle_{H^{-\frac{1}{2}-\chi}(\partial\Omega),H^{\frac{1}{2}+\chi}(\partial\Omega)}-\left\langle \rho^{-1} F(0,\cdot),\phi_{\chi,n}\right\rangle_{-1}\right)t^{\alpha-1}E_{\alpha,\alpha}(-\lambda_{\chi,n}t^\alpha)
$$
and, combining this with \eqref{t3c}, we obtain 
$$\begin{aligned}
u_n'(t)=&-(-1)^\chi\int_0^t(t-s)^{\alpha-1}E_{\alpha,\alpha}(-\lambda_{\chi,n}(t-s)^\alpha)\left\langle \partial_sf(s,\cdot),\tau_{\chi}^*\phi_{\chi,n}\right\rangle_{H^{-\frac{1}{2}-\chi}(\partial\Omega),H^{\frac{1}{2}+\chi}(\partial\Omega)}ds\\
&+\int_0^ts^{\alpha-1}E_{\alpha,\alpha}(-\lambda_{\chi,n}s^\alpha)\left\langle \rho^{-1} \partial_tF(t-s,\cdot),\phi_{\chi,n}\right\rangle_{-1}ds.
\end{aligned}$$
Thus, repeating the arguments in the proof of Theorem \ref{t1}, we deduce that
$$
\sum_{k=1}^nu_k'(t)\phi_{\chi,k},\quad n\in\mathbb N
$$
converges in $L^r(0,T;D(A_{\chi}^{-\epsilon}))$ and
$$
\norm{\sum_{k=1}^\infty u_k'\phi_{\chi,k}}_{L^r(0,T;D(A_{\chi}^{-\epsilon}))}
\leq C_\epsilon (\norm{\partial_tf}_{L^r(0,T;H^{-\frac{1}{2}-\chi}(\partial\Omega))}+\norm{\rho^{-1}\partial_tF}_{L^r(0,T;D(A_{\chi}^{-1}))}).
$$
Therefore $u\in W^{1,r}(0,T;D(A_{\chi}^{-\epsilon}))$ and
$$ 
\norm{u}_{W^{1,r}(0,T;D(A_{\chi}^{-\epsilon}))}\leq C_\epsilon \left(\norm{f}_{W^{1,r}(0,T;H^{-\frac{1}{2}-\chi}(\partial\Omega))}+\norm{\rho^{-1}F}_{W^{1,r}(0,T;D(A_{\chi}^{-1})}\right).
$$
Moreover, in view of \cite[Theorem 11.1, Chapter 1]{LM1}, 
since $D(A_{\chi}^{\epsilon})= H^{2\epsilon}_0(\Omega)$ for
$\epsilon\in(0,1/4)$ with the equivalent norms, the duality 
yields $H^{-2\epsilon}(\Omega)=D(A_{\chi}^{-\epsilon})$.
Thus, the last inequality  can be rewritten as
\bel{t3d} 
\norm{u}_{W^{1,r}(0,T;H^{-2\epsilon}(\Omega))}\leq C_\epsilon \left(\norm{f}_{W^{1,r}(0,T;H^{-\frac{1}{2}-\chi}(\partial\Omega))}+\norm{\rho^{-1}F}_{W^{1,r}(0,T;D(A_{\chi}^{-1})}\right).
\ee
In order to complete the proof for $m=1$, we need to prove that $u\in L^r(0,T;H^{2}(\Omega))$ satisfies
\bel{t3e} 
\norm{u}_{L^r(0,T;H^{2}(\Omega))}\leq C \left(\sum_{k=0}^1\norm{f}_{W^{1-k,r}(0,T;H^{2k-\frac{1}{2}-\chi}(\partial\Omega))}+\norm{\rho^{-1}F}_{W^{1-k,r}(0,T;D(A_{\chi}^{k-1}))}\right).
\ee
Using \cite[Lemma 3.2]{SY} and integrating by parts, 
for almost all $t\in(0,T)$, one can verify that
$$\begin{aligned}
u_{1,n}(t)&=-(-1)^\chi\int_0^ts^{\alpha-1}E_{\alpha,\alpha}(-\lambda_{\chi,n}s^\alpha)\left\langle f(t-s,\cdot),\tau_{\chi}^*\phi_{\chi,n}\right\rangle_{L^2(\partial\Omega)}ds\\
\ &=(-1)^\chi\int_0^t\partial_s\left(\frac{E_{\alpha,1}(-\lambda_{\chi,n}s^\alpha)}{\lambda_{\chi,n}}\right)\left\langle f(t-s,\cdot),\tau_{\chi}^*\phi_{\chi,n}\right\rangle_{L^2(\partial\Omega)}ds\\
\ &=(-1)^\chi\int_0^tE_{\alpha,1}(-\lambda_{\chi,n}s^\alpha)\frac{\left\langle \partial_tf(t-s,\cdot),\tau_{\chi}^*\phi_{\chi,n}\right\rangle_{H^{-\frac{1}{2}-\chi}(\partial\Omega),H^{\frac{1}{2}+\chi}(\partial\Omega)}}{\lambda_{\chi,n}}ds
\\
\ &\ \ \  +(-1)^\chi E_{\alpha,1}(-\lambda_{\chi,n}t^\alpha)\frac{\left\langle f(0,\cdot),\tau_{\chi}^*\phi_{\chi,n}\right\rangle_{H^{-\frac{1}{2}-\chi}(\partial\Omega),H^{\frac{1}{2}+\chi}(\partial\Omega)}}{\lambda_{\chi,n}}\\
&\ \ \ -(-1)^\chi\frac{\left\langle f(t,\cdot),\tau_{\chi}^*\phi_{\chi,n}\right\rangle_{L^2(\partial\Omega)}}{\lambda_{\chi,n}}.
\end{aligned}$$
Similarly we find
$$\begin{aligned}
u_{3,n}(t)=&\int_0^tE_{\alpha,1}(-\lambda_{\chi,n}s^\alpha)\frac{\left\langle \rho^{-1}\partial_tF(t-s,\cdot),\phi_{\chi,n}\right\rangle
_{-1}}{\lambda_{\chi,n}}ds\\
\ &-E_{\alpha,1}(-\lambda_{\chi,n}t^\alpha)\frac{\left\langle \rho^{-1}F(0,\cdot),\phi_{\chi,n}\right\rangle_{-1}}{\lambda_{\chi,n}} +\frac{\left\langle \rho^{-1} F(t,\cdot),\phi_{\chi,n}\right\rangle}{\lambda_{\chi,n}}.
\end{aligned}$$
Thus, applying again the compatibility condition \eqref{com1}, we find
\bel{t3f} \begin{aligned}
u_n(t)=&(-1)^\chi\int_0^tE_{\alpha,1}(-\lambda_{\chi,n}s^\alpha)\frac{\left\langle \partial_tf(t-s,\cdot),\tau_{\chi}^*\phi_{\chi,n}\right\rangle_{H^{-\frac{1}{2}-\chi}(\partial\Omega),H^{\frac{1}{2}+\chi}(\partial\Omega)}}{\lambda_{\chi,n}}ds\\
&+\int_0^tE_{\alpha,1}(-\lambda_{\chi,n}s^\alpha)\frac{\left\langle \rho^{-1}\partial_tF(t-s,\cdot),\phi_{\chi,n}\right\rangle_{-1}}{\lambda_{\chi,n}}ds\\
&+\frac{(-1)^{\chi+1}\left\langle f(t,\cdot),\tau_{\chi}^*\phi_{\chi,n}\right\rangle_{L^2(\partial\Omega)}+\left\langle \rho^{-1} F(t,\cdot),\phi_{\chi,n}\right\rangle}{\lambda_{\chi,n}}.
\end{aligned}\ee
We set  
$$
w_{1,n}(t):=(-1)^\chi\int_0^tE_{\alpha,1}(-\lambda_{\chi,n}s^\alpha)\frac{\left\langle \partial_tf(t-s,\cdot),\tau_{\chi}^*\phi_{\chi,n}\right\rangle_{H^{-\frac{1}{2}-\chi}(\partial\Omega),H^{\frac{1}{2}+\chi}(\partial\Omega)}}{\lambda_{\chi,n}}ds,
$$
$$
w_{2,n}(t):=\int_0^tE_{\alpha,1}(-\lambda_{\chi,n}s^\alpha)\frac{\left\langle \rho^{-1}\partial_tF(t-s,\cdot),\phi_{\chi,n}\right\rangle_{-1}}{\lambda_{\chi,n}}ds,
$$
and
$$
w_{3,n}(t):=\frac{(-1)^{\chi+1}\left\langle f(t,\cdot),\tau_{\chi}^*\phi_{\chi,n}\right\rangle_{H^{-\frac{1}{2}-\chi}(\partial\Omega),H^{\frac{1}{2}+\chi}(\partial\Omega)}+\left\langle \rho^{-1} F(t,\cdot),\phi_{\chi,n}\right\rangle}{\lambda_{\chi,n}}.
$$
In view of \eqref{t3f}, the condition
\bel{t3g}
\sum_{n=1}^\infty w_{i,n}\phi_{\chi,n}\in L^r(0,T;H^2(\Omega)),\quad 
i=1,2,3,
\ee
implies that $u\in L^r(0,T;H^{2}(\Omega))$. 
Moreover the estimate
\bel{t3h} \begin{aligned}
&\norm{\sum_{n=1}^\infty w_{i,n}\phi_{\chi,n}}_{L^r(0,T;H^{2}(\Omega))}\\
&\leq C \left(\sum_{k=0}^1\norm{f}_{W^{1-k,r}(0,T;H^{2k-\frac{1}{2}-\chi}(\partial\Omega))}+\norm{\rho^{-1}F}_{W^{1-k,r}(0,T;D(A_{\chi}^{k-1}))}\right),
\quad i=1,2,3,
\end{aligned}\ee
yields \eqref{t3e}. In order to complete the proof of our result for $m=1$, 
we have to prove \eqref{t3g} and \eqref{t3h}.  Applying formula (1.148) of 
\cite[Theorem 1.6]{P}, similarly to the proof of Theorem \ref{t1}, 
for $m<n$ we have 
$$\begin{aligned}
&\norm{\sum_{\ell=m}^n w_{1,\ell}\phi_{\chi,\ell}}_{L^r(0,T;D(A_{\chi}))}\\
&\leq C\norm{\int_0^ts^{-\alpha}\norm{\sum_{\ell=m}^n\frac{\left\langle \partial_tf(t-s,\cdot),\tau_{\chi}^*\phi_{\chi,\ell}\right\rangle_{H^{-\frac{1}{2}-\chi}(\partial\Omega),H^{\frac{1}{2}+\chi}(\partial\Omega)}}{\lambda_{\chi,\ell}}\phi_{\chi,\ell}}_{L^2(\Omega)}ds}_{L^r(0,T)}\\
&\leq C\norm{\sum_{\ell=m}^n\frac{\left\langle \partial_tf(t-s,\cdot),\tau_{\chi}^*\phi_{\chi,\ell}\right\rangle_{H^{-\frac{1}{2}-\chi}(\partial\Omega),H^{\frac{1}{2}+\chi}(\partial\Omega)}}{\lambda_{\chi,\ell}}\phi_{\chi,\ell}}
_{L^r(0,T;L^2(\Omega))}.
\end{aligned}$$
On the other hand, in view of Lemma \ref{l1} and 
$\partial_tf\in L^r(0,T;H^{-\frac{1}{2}-\chi}(\partial\Omega))$, the sequence 
$$
\sum_{\ell=1}^N\frac{\left\langle \partial_tf(t-s,\cdot),\tau_{\chi}^*\phi_{\chi,\ell}\right\rangle_{H^{-\frac{1}{2}-\chi}(\partial\Omega),H^{\frac{1}{2}+\chi}(\partial\Omega)}}{\lambda_{\chi,\ell}},\quad N\in\mathbb N
$$
converges in $L^r(0,T;L^2(\Omega))$. 
Therefore we have
$$
\lim_{m,n\to\infty}\norm{\sum_{\ell=m}^n\frac{\left\langle \partial_tf(t-s,\cdot),\tau_{\chi}^*\phi_{\chi,\ell}\right\rangle_{H^{-\frac{1}{2}-\chi}(\partial\Omega),H^{\frac{1}{2}+\chi}(\partial\Omega)}}{\lambda_{\chi,\ell}}\phi_{\chi,\ell}}_{L^r(0,T;L^2(\Omega))}=0,
$$
which implies that the sequence \eqref{t3g} with $i=1$ converges 
in $L^r(0,T;D(A_{\chi}))$. 
Moreover, since $D(A_{\chi})$ is embedded continuously into $H^2(\Omega)$, 
it follows that \eqref{t3g} with $i=1$ holds true.  The proof for
\eqref{t3g} with $i=2$ is similar and so omitted.

Using similar arguments again, we obtain estimate \eqref{t3h} for $i=1,2$. 
Now let $y$ satisfy $y(t,\cdot)\in H^1(\Omega)$ for almost all $t\in(0,T)$
and satisfy  
\bel{eleq}
\left\{ \begin{array}{rcll} 
 \cA y(t,x)& = & F(t,x), & x\in   \Omega,\\
\tau_{\chi} y(t,x) & = & f(t,\cdot), & x \in  \pd \Omega. 
\end{array}
\right.\ee
Combining $f\in L^r(0,T;H^{\frac{3}{2}-\chi}(\partial\Omega))$ and
$F\in L^r(0,T;L^2(\Omega))$ with the elliptic regularity of the operator $\cA$,
we have that $y\in L^r(0,T;H^2(\Omega))$ and 
\bel{t3i} 
\norm{y}_{L^r(0,T;H^2(\Omega))}\leq C(\norm{f}_{L^r(0,T;H^{\frac{3}{2}-\chi}(\partial\Omega))}+\norm{F}_{L^r(0,T;L^2(\Omega))}).
\ee
Moreover, following the proof of Lemma \ref{l1}, one can verify that
\bel{eleq1}
\left\langle y(t,\cdot),\phi_{\chi,n}\right\rangle=\frac{(-1)^{\chi+1}\left\langle f(t,\cdot),\tau_{\chi}^*\phi_{\chi,n}\right\rangle_{H^{-\frac{1}{2}-\chi}(\partial\Omega),H^{\frac{1}{2}+\chi}(\partial\Omega)}+\left\langle \rho^{-1} F(t,\cdot),\phi_{\chi,n}\right\rangle}{\lambda_{\chi,n}}=w_{3,n}(t).
\ee
Thus, we have
$$
\sum_{n=1}^\infty w_{3,n}(t)\phi_{\chi,n}=y(t,\cdot)
$$
and by $y\in L^r(0,T;H^2(\Omega))$, we deduce \eqref{t3g} with $i=3$. In addition, we obtain \eqref{t3h} for $i=3$ from estimate \eqref{t3i}. This proves that $u\in L^r(0,T;H^2(\Omega))$. Therefore we prove that
$u\in W^{1,r}(0,T;H^{-\epsilon}(\Omega))\cap L^r(0,T;H^2(\Omega))$ 
for all $\epsilon\in(0,1/2)$ and $u$ satisfies \eqref{t3a}. 
In order to complete the proof of Theorem \ref{t3},
it suffices to verify that $u$ is a strong solution to \eqref{eq1} 
in the sense of Definition \ref{d3}. For this purpose, let 
$$
\widetilde{f}\in W^{1,r}(0,+\infty;H^{-\frac{1}{2}-\chi}(\pd\Omega))\cap L^r(0,+\infty;H^{\frac{3}{2}-\chi}(\pd\Omega)),
$$
$$
\widetilde{F}\in W^{1,r}(0,+\infty;\rho D(A_{\chi}^{-1}))\cap L^r(0,+\infty;L^2(\Omega)),
$$
satisfy supp $\widetilde{f} \subset[0,T+1)\times\pd \Omega$, 
supp $\widetilde{F}\subset[0,T+1)\times\overline{\Omega}$ and 
\bel{t3k}
\widetilde{f}|_\Sigma=f,\quad \widetilde{F}|_Q=F.
\ee
Now we set 
$$
\widetilde{v}(t,\cdot):=\sum_{n=1}^\infty \widetilde{v}_n(t)\phi_{\chi,n}
$$
with 
$$\begin{aligned}
\widetilde{v}_n(t):=&-(-1)^\chi\int_0^t(t-s)^{\alpha-1}E_{\alpha,\alpha}(-\lambda_{\chi,n}(t-s)^\alpha)\left\langle \widetilde{f}(s,\cdot),\tau_{\chi}^*\phi_{\chi,n}\right\rangle_{L^2(\partial\Omega)}ds\\
&+\int_0^t(t-s)^{\alpha-1}E_{\alpha,\alpha}(-\lambda_{\chi,n}(t-s)^\alpha)\left\langle \widetilde{F}(s,\cdot),\phi_{\chi,n}\right\rangle_{L^2(\Omega)}ds\\
\ &+E_{\alpha,1}(-\lambda_{\chi,n}t^\alpha)\left\langle  u_0,\phi_{\chi,n}\right\rangle.
\end{aligned}$$
From the above arguments, one can verify that $\widetilde{v}\in W^{1,r}_{loc}(0,+\infty;H^{-\epsilon}(\Omega))\cap L^r_{loc}(0,+\infty;H^2(\Omega))$ 
and 
\bel{t3j} 
e^{-pt}\widetilde{v}(t,\cdot)\in W^{1,1}(0,+\infty;H^{-\epsilon}(\Omega))\cap L^1(0,+\infty;H^2(\Omega))
\ee
for $p>0$.
Moreover, following the proof of Theorem \ref{t3}, we see 
that the Laplace transform $\mathcal{L}\widetilde{v}(p,x)$ of $\widetilde{v}$ 
satisfies 
$$
p^\alpha(\mathcal{L}\widetilde{v}(p,x)
-p^{-1}u_0(x))+\rho^{-1}\mathcal A \mathcal{L}\widetilde{v}(p,x)
=\int_0^{T+1}e^{-pt}\rho^{-1}\widetilde{F}(t,x)dt,\quad (p,x)\in 
(0,+\infty) \times\Omega.
$$
Henceforth we write $D'(\Omega) = \mathcal{C}^{\infty}_0(\Omega)'$.
In view of \eqref{t3j}, this identity implies that, for all $p>0$ and all $\psi\in \mathcal C^\infty_0(\Omega)$, we have
$$\begin{aligned}
&\mathcal L\left[ \left\langle (\rho\partial_t^\alpha \widetilde{v}+\mathcal A \widetilde{v})(t,\cdot),\psi\right\rangle_{D'(\Omega),\mathcal C^\infty_0(\Omega)}\right](p)\\
&=\left\langle\mathcal L[ (\rho\partial_t^\alpha \widetilde{v}+\mathcal A \widetilde{v})](p,\cdot),\psi\right\rangle_{D'(\Omega),\mathcal C^\infty_0(\Omega)}\\
&=\left\langle\rho p^\alpha(\mathcal{L}\widetilde{v}(p,\cdot)-u_0)
+ \mathcal A \mathcal{L}\widetilde{v}(p,\cdot),\psi
\right\rangle_{D'(\Omega),\mathcal C^\infty_0(\Omega)}\\
&=\left\langle\int_0^{T+1}e^{-pt}\widetilde{F}(t,\cdot)dt,\psi\right\rangle_{D'(\Omega),\mathcal C^\infty_0(\Omega)}\\
&=\mathcal L\left[ \left\langle \widetilde{F}(t,\cdot),\psi\right\rangle_{D'(\Omega),\mathcal C^\infty_0(\Omega)}\right](p).
\end{aligned}$$
Therefore, for almost all $t>0$, we have
$$
\left\langle (\rho\partial_t^\alpha \widetilde{v}+\mathcal A \widetilde{v})(t,\cdot)-\widetilde{F}(t,\cdot),\psi\right\rangle_{D'(\Omega),\mathcal C^\infty_0(\Omega)}=0.
$$
On the other hand, in view of \eqref{t3k} one can easily verify that
$\widetilde{v}=u$ in $(0,T)\times\Omega$.  Hence for almost all $t\in(0,T)$
we have
$$
\left\langle (\rho\partial_t^\alpha u+\mathcal A u)(t,\cdot)-F(t,\cdot),\psi\right\rangle_{D'(\Omega),\mathcal C^\infty_0(\Omega)}=0,\quad \psi\in\mathcal C^\infty_0(\Omega).
$$
Furthermore, by $u\in L^r(0,T;H^2(\Omega))$, we have
$\partial_t^\alpha u=-\rho^{-1}\mathcal A u+F\in L^r(0,T;L^2(\Omega))$.
Therefore \eqref{id} holds true in $L^r(0,T;L^2(\Omega))$. In the same way, 
applying \eqref{t3j}, one can verify that condition \eqref{id1} is also 
fulfilled and $u$ is a strong solution to \eqref{eq1}.
This completes the proof of Theorem \ref{t3} for $m=1$.

\subsection{For $m\geq 2$}

We will consider only the case $m=2$. The case $m\geq3$ can be deduced in a similar way by an iteration argument. Applying the result for $m=1$, for all $\epsilon\in(0,1/2)$,  we obtain that $u\in W^{1,r}(0,T;H^{-\epsilon}(\Omega))\cap L^r(0,T;H^2(\Omega))$ and 
$$
\partial_tu(t,\cdot) = \sum_{n=1}^\infty u_n'(t)\phi_{\chi,n},
$$
where 
\bel{t3m}\begin{aligned}
u_n'(t)=&-(-1)^\chi\int_0^t(t-s)^{\alpha-1}E_{\alpha,\alpha}(-\lambda_{\chi,n}(t-s)^\alpha)\left\langle \partial_sf(s,\cdot),\tau_{\chi}^*\phi_{\chi,n}\right\rangle_{L^2(\partial\Omega)}ds\\
&+\int_0^t(t-s)^{\alpha-1}E_{\alpha,\alpha}(-\lambda_{\chi,n}(t-s)^\alpha)\left\langle \rho^{-1} \partial_sF(s,\cdot),\phi_{\chi,n}\right\rangle ds.
\end{aligned}\ee
Applying \eqref{tt3a}, we obtain
$$\begin{aligned}
u_n''(t)=&-(-1)^\chi\int_0^ts^{\alpha-1}E_{\alpha,\alpha}(-\lambda_{\chi,n}s^\alpha)\left\langle \partial_s^2f(s,\cdot),\tau_{\chi}^*\phi_{\chi,n}\right\rangle_{H^{-\frac{1}{2}-\chi}(\partial\Omega), H^{\frac{1}{2}+\chi}(\partial\Omega)}ds
                                                    \\
&+\int_0^t(t-s)^{\alpha-1}E_{\alpha,\alpha}(-\lambda_{\chi,n}(t-s)^\alpha)\left\langle \rho^{-1} \partial_s^2F(s,\cdot),\phi_{\chi,n}\right\rangle_{-1}ds
\end{aligned}.$$
and, following the proof for $m=1$, we reach $u\in W^{2,r}(0,T;H^{-\epsilon}(\Omega))$ and 
\bel{t3l} 
\norm{u}_{W^{2,r}(0,T; H^{-\epsilon}(\Omega))}\leq C_\epsilon \left(\norm{f}_{W^{2,r}(0,T; H^{-\frac{1}{2}-\chi}(\pd\Omega))}+\norm{F}_{W^{2,r}(0,T; D(A_{\chi}^{-\frac{1}{2}}))}\right).
\ee
Now let us prove that $\partial_t u\in L^r(0,T;H^2(\Omega))$. For this purpose, in view of \eqref{t3m}, applying the arguments for the proof of 
Theorem \ref{t1}, we see that $v=\partial_tu$ is the weak solution to
$$\left\{ \begin{array}{rcll} 
(\rho(x) \partial_t^{\alpha}+\cA ) v(t,x) & = & \partial_tF(t,x), & (t,x)\in  Q,\\
\tau_{\chi} v(t,x) & = & \partial_tf(t,x), & (t,x) \in \Sigma, \\  
 v(0,x) & = & 0, & x \in \Omega.
\end{array}
\right.
$$
Since 
$$
\partial_tf\in L^r(0,T;H^{\frac{3}{2}-\chi}(\partial\Omega))\cap 
W^{1,r}(0,T;H^{-\frac{1}{2}-\chi}(\partial\Omega)), \quad
\partial_tF\in L^r(0,T;L^2(\Omega))\cap W^{1,r}(0,T;\rho D(A_{\chi}^{-1}))
$$
and \eqref{tt3a} is fulfilled, applying the result of the theorem for $m=1$, we obtain 
$\partial_tu\in L^r(0,T;H^2(\Omega))\cap W^{1,r}(0,T;H^{-\epsilon}(\Omega))$ 
and
\bel{t3n} \begin{aligned}
&\norm{u}_{W^{1,r}(0,T; H^{2}(\Omega))}+\norm{u}_{W^{2,r}(0,T; H^{-\epsilon}
(\Omega))}\\
&\leq C_\epsilon \left(\sum_{k=1}^2\norm{f}_{W^{2-k,r}(0,T; H^{2k-\frac{1}{2}-\chi}(\pd\Omega))}+\norm{F}_{ W^{2-k,r}(0,T; H^{2(k-1)}(\Omega))}\right)\\
& \ \ \ +C_\epsilon\left(\norm{f}_{W^{2,r}(0,T; H^{-\frac{1}{2}-\chi}(\pd\Omega))}+\norm{\rho^{-1}F}_{ W^{2,r}(0,T; D(A_{\chi}^{-1}))}\right).
\end{aligned}\ee
Recalling that $u$ is a strong solution to \eqref{eq1}, we see 
that $u(t,\cdot)$ solves the boundary value problem
$$
\left\{ \begin{array}{rcll} 
 \cA u(t,x)& = & -\rho(x)\partial_t^\alpha u(t,x)+F(t,x), & x\in   \Omega,\\
\tau_{\chi} u(t,x) & = & f(t,x), & x \in  \pd \Omega 
\end{array}
\right.
$$
for almost all $t\in(0,T)$.  
On the other hand, since $u\in W^{1,r}(0,T; H^{2}(\Omega))$, 
Young's inequality for convolution yields 
$\partial_t^\alpha u\in L^r(0,T;H^{2}(\Omega))$. 
Therefore, by $f\in L^r(0,T; H^{2+\frac{3}{2}-\chi}(\pd\Omega))$ and
$F\in L^r(0,T;H^2(\Omega))$, applying the elliptic regularity of 
the operator $ \cA$ guaranteed by condition $(H_2)$, we see that 
$u\in L^r(0,T;H^4(\Omega))$ and  
$$ \begin{aligned}
&\norm{u}_{L^r(0,T; H^4(\Omega))}\\
&\leq C(\norm{u}_{W^{1,r}(0,T; H^{2}(\Omega))}+\norm{f}_{L^r(0,T; H^{2+\frac{3}{2}-\chi}(\pd\Omega))}+\norm{F}_{L^r(0,T;H^2(\Omega))}).
\end{aligned}$$
Combining this with \eqref{t3n}, we find
$$\begin{aligned}
&\norm{u}_{W^{1,r}(0,T; H^{2}(\Omega))}+\norm{u}_{W^{2,r}(0,T; H^{-\epsilon}(\Omega))}+\norm{u}_{L^r(0,T; H^4(\Omega))}\\
& \leq C_\epsilon \sum_{k=1}^2\left(\norm{f}_{W^{2-k,r}(0,T; H^{2k-\frac{1}{2}-\chi}(\pd\Omega))}+\norm{F}_{ W^{2-k,r}(0,T; H^{2(k-1)}(\Omega))}\right)\\
&\ \ \  +C_\epsilon\left(\norm{f}_{W^{2,r}(0,T; H^{-\frac{1}{2}-\chi}(\pd\Omega))}+\norm{\rho^{-1}F}_{ W^{2,r}(0,T; D(A_{\chi}^{-1}))}\right).
\end{aligned}$$
This completes the proof of Theorem \ref{t3}

\section{Proof of Theorem \ref{t4}}

For all $n\in\mathbb N$, we set 
$$
u_{1,n}(t):=-(-1)^\chi\int_0^t(t-s)^{\alpha-1}E_{\alpha,\alpha}(-\lambda_{\chi,n}(t-s)^\alpha)\left\langle f(s,\cdot),\tau_{\chi}^*\phi_{\chi,n}\right\rangle_{L^2(\partial\Omega)}ds,
$$
$$
u_{2,n}(t):=E_{\alpha,1}(-\lambda_{\chi,n}t^\alpha)\left\langle u_0,\phi_{\chi,n}\right\rangle,
$$
$$ 
u_{3,n}(t):=tE_{\alpha,2}(-\lambda_{\chi,n}t^\alpha)\left\langle  u_1,\phi_{\chi,n}\right\rangle,
$$
$$ 
u_{4,n}(t):=\int_0^t(t-s)^{\alpha-1}E_{\alpha,\alpha}(-\lambda_{\chi,n}(t-s)^\alpha)\mathds{1}_{(0,T)}(s)\left\langle \rho^{-1}F(s,\cdot),\phi_{\chi,n}\right\rangle ds
$$
and
$$
u_n:=u_{1,n}+u_{2,n}+u_{3,n}+u_{4,n}.
$$
Let $\epsilon>0$.
By means of Theorem \ref{t1}, we know that the sequence
$$
\sum_{n=1}^N u_n\phi_{\chi,n}, \quad N \in \mathbb N,
$$
converges in $L^r(0,T;D(A_{\chi}^{-\epsilon}))$ 
to the unique solution $u$ to \eqref{eq1}. 
Let us show that $u\in L^r(0,T;L^2(\Omega))\cap 
W^{2,r}(0,T;L^2(\Omega))$ and 
\bel{t4c}
\norm{u}_{W^{2,r}(0,T;L^2(\Omega))}\leq C\left(\norm{f}_{W^{2,r}(0,T;H^{\frac{1}{2}-\chi}(\partial\Omega))}+\norm{F}_{W^{2,r}(0,T;D(A_{\chi}^{-\frac{1}{2}}))}+\norm{u_1}_{H^{2\delta}(\Omega)}\right).\ee
Let us first prove that 
\bel{cc}
\lim_{N\to \infty} \sum_{n=1}^Nu_n\phi_{\chi,n},
\ee
converges in $L^r(0,T;L^2(\Omega))$ to $u$. Since $\rho^{-1}\in\mathcal C^1(\overline{\Omega})$, one sees that $\rho^{-1}F\in W^{2,r}(0,T;D(A_{\chi}^{-\frac{1}{2}}))$ and
\bel{ttt}
\norm{\rho^{-1}F}_{W^{k,r}(0,T;D(A_{\chi}^{-\frac{1}{2}}))}\leq C\norm{F}_{W^{2,r}(0,T;D(A_{\chi}^{-\frac{1}{2}}))},\quad k=0,1,2.
\ee
Combining this with  the decay property of the Mittag-Leffler functions given by formula (1.148) of \cite[Theorem 1.6]{P}, the convergence of \eqref{cc} 
follows from the convergence of
$$
\lim_{N\to\infty} \sum_{n=1}^Nu_{1,n}\phi_{\chi,n}
$$
in $L^r(0,T;L^2(\Omega))$. 
For this purpose, let $y\in L^r(0,T;L^2(\Omega))$ satisfy 
\eqref{eleq} with $F=0$. Since $f\in L^r(0,T;H^{\frac{1}{2}-\chi}(\partial\Omega))$, we have $y\in L^r(0,T;H^1(\Omega))$ and
$$
\norm{y}_{L^r(0,T;H^1(\Omega))}\leq C\norm{f}_{L^r(0,T;H^{\frac{1}{2}-\chi}(\partial\Omega))}.
$$
The same arguments yield  
$$
\norm{y}_{L^r(0,T;H^{\frac{1}{4}}(\Omega))}\leq C\norm{f}_{L^r(0,T;H^{\frac{1}{2}-\chi}(\partial\Omega))}.
$$
In view of \cite[Theorem 11.1, Chapter 1]{LM1}, we have $H^{\frac{1}{4}}(\Omega)=H^{\frac{1}{4}}_0(\Omega)=D(A_{\chi}^{\frac{1}{8}})$. 
Combination of this with \eqref{eleq1} implies 
\bel{t4d}
\norm{\left( \sum_{n=1}^\infty \abs{\lambda_{\chi,n}^{-\frac{7}{8}}
\left\langle f(s,\cdot),\tau_{\chi}^*\phi_{\chi,n}\right\rangle_{L^2(\partial\Omega)}}^2\right)^{1/2}}_{L^r(0,T)}
\leq C \norm{f}_{L^r(0,T;H^{\frac{1}{2}-\chi}(\partial\Omega))}.
\ee
Therefore, by formula (1.148) in \cite[Theorem 1.6]{P} and \eqref{t4d}, 
we have 
$$\begin{aligned}
&\limsup_{m,n\to\infty}\norm{\sum_{k=m}^nu_{1,k}\phi_{\chi,k}}_{L^r(0,T;L^2(\Omega))}\\
&\leq C\limsup_{m,n\to\infty}\norm{\int_0^t(t-s)^{\frac{\alpha}{8}-1}\norm{\sum_{k=m}^n \abs{\lambda_{\chi,k}^{-\frac{7}{8}}\left\langle f(s,\cdot),\tau_{\chi}^*\phi_{\chi,k}\right\rangle_{L^2(\partial\Omega)}}\phi_{\chi,k}}
_{L^2(\Omega)}}_{L^r(0,T)}\\
\ &\leq C\limsup_{m,n\to\infty}\norm{\sum_{k=m}^n \abs{\lambda_{\chi,k}^{-\frac{7}{8}}\left\langle f(s,\cdot),\tau_{\chi}^*\phi_{\chi,k}\right\rangle_{L^2(\partial\Omega)}}\phi_{\chi,k}}_{L^r(0,T;L^2(\Omega))}=0.
\end{aligned}$$
Thus, the sequence
$$
\sum_{n=1}^Nu_{1,n}\phi_{\chi,n},\quad N\in\mathbb N
$$
is a Cauchy sequence and so a convergent sequence in 
$L^r(0,T;L^2(\Omega))$. Therefore, we have $u\in L^r(0,T;L^2(\Omega))$. 
On the other hand, thanks to \eqref{com1}, we have
$$\begin{aligned}
\norm{u_0}_{H^1(\Omega)}&\leq C(\norm{f(0,\cdot)}_{H^{\frac{1}{2}-\chi}
(\partial\Omega))}+\norm{\rho^{-1}F(0,\cdot)}_{H^{-1}(\Omega)})\\
\ &\leq C\left(\norm{f}_{W^{2,r}(0,T;H^{\frac{1}{2}-\chi}(\partial\Omega))}
+\norm{F}_{W^{1,r}(0,T;D(A_{\chi}^{-\frac{1}{2}}))}\right).
\end{aligned}$$
Hence, from \eqref{ttt} and \eqref{t4d} we obtain
\bel{t4e}
\norm{u}_{L^r(0,T;L^2(\Omega))}\leq C\left(\norm{f}_{W^{2,r}(0,T;H^{\frac{1}{2}-\chi}(\partial\Omega))}+\norm{F}_{W^{2,r}(0,T;D(A_{\chi}^{-\frac{1}{2}}))}+\norm{u_1}_{H^{2\delta}(\Omega)}\right).
\ee
Here we notice that since $D(A_{\chi}^{-\frac{1}{2}})$ is
embedded continuously into $H^{-1}(\Omega)$ by the duality, because 
$H^1_0(\Omega)$ is embedded continuously into $D(A_{\chi}^{\frac{1}{2}})$.

Now let us show that $u\in W^{1,r}(0,T;L^2(\Omega))$ and 
\bel{t4f}
\norm{u}_{W^{1,r}(0,T;L^2(\Omega))}\leq C\left(\norm{f}_{W^{2,r}(0,T;H^{\frac{1}{2}-\chi}(\partial\Omega))}+\norm{F}_{W^{2,r}(0,T;H^{-1}(\partial\Omega))}+\norm{u_1}_{H^{2\delta}(\Omega)}\right).
\ee
By \cite[Lemma 3.2]{SY}, we have
$$\begin{aligned}
u_{1,n}'(t)=&-(-1)^\chi\int_0^ts^{\alpha-1}E_{\alpha,\alpha}(-\lambda_{\chi,n}s^\alpha)\left\langle \partial_tf(t-s,\cdot),\tau_{\chi}^*\phi_{\chi,n}\right\rangle_{L^2(\partial\Omega)}ds\\
\ &-(-1)^\chi \left\langle f(0,\cdot),\tau_{\chi}^*\phi_{\chi,n}\right\rangle_{H^{-\frac{1}{2}-\chi}(\partial\Omega),H^{\frac{1}{2}+\chi}(\partial\Omega)} t^{\alpha-1}E_{\alpha,\alpha}(-\lambda_{\chi,n}t^\alpha),\end{aligned}
$$
$$
u_{2,n}'(t)=-\lambda_{\chi,n}\left\langle u_0,\phi_{\chi,n}\right\rangle t^{\alpha-1}E_{\alpha,\alpha}(-\lambda_{\chi,n}t^\alpha),
$$
$$ 
u_{3,n}'(t)=E_{\alpha,1}(-\lambda_{\chi,n}t^\alpha)\left\langle  u_1,\phi_{\chi,n}\right\rangle
$$
and
$$\begin{aligned}
u_{4,n}'(t)=&\int_0^ts^{\alpha-1}E_{\alpha,\alpha}(-\lambda_{\chi,n}s^\alpha)\left\langle \rho^{-1}\partial_tF(t-s,\cdot),\phi_{\chi,n}\right\rangle_{-\frac{1}{2}}ds\\
\ &+ \left\langle \rho^{-1}F(0,\cdot),\phi_{\chi,n}\right\rangle_{-\frac{1}{2}} t^{\alpha-1}E_{\alpha,\alpha}(-\lambda_{\chi,n}t^\alpha)
\end{aligned}
$$
for almost all $t\in(0,T)$. 
In a similar way to the proof of Theorem \ref{t3}, we deduce that
$$
u_{2,n}'(t)=\left[(-1)^\chi \left\langle f(0,\cdot),\tau_{\chi}^*\phi_{\chi,n}\right\rangle_{H^{-\frac{1}{2}-\chi}(\partial\Omega),H^{\frac{1}{2}+\chi}(\partial\Omega)}-\left\langle \rho^{-1}F(0,\cdot),\phi_{\chi,n}\right\rangle_{-\frac{1}{2}}\right] t^{\alpha-1}E_{\alpha,\alpha}(-\lambda_{\chi,n}t^\alpha)
$$
and so 
$$\begin{aligned}
u_n'(t)=&-(-1)^\chi\int_0^ts^{\alpha-1}E_{\alpha,\alpha}(-\lambda_{\chi,n}s^\alpha)\left\langle \partial_tf(t-s,\cdot),\tau_{\chi}^*\phi_{\chi,n}\right\rangle_{H^{-\frac{1}{2}-\chi}(\partial\Omega),H^{\frac{1}{2}+\chi}(\partial\Omega)}ds\\
&+\int_0^ts^{\alpha-1}E_{\alpha,\alpha}(-\lambda_{\chi,n}s^\alpha)\left\langle \rho^{-1}\partial_tF(t-s,\cdot),\phi_{\chi,n}\right\rangle_{-\frac{1}{2}}ds\\
\ &+E_{\alpha,1}(-\lambda_{\chi,n}t^\alpha)\left\langle  u_1,\phi_{\chi,n}\right\rangle.
\end{aligned}$$
Therefore, in view of 
$\partial_tf\in L^r(0,T;H^{\frac{1}{2}-\chi}(\partial\Omega))$
and $\partial_tF\in L^r(0,T;D(A_{\chi}^{-\frac{1}{2}}))$,
repeating the arguments used for proving that $u\in L^r(0,T;L^2(\Omega))$, 
we see that $u\in W^{1,r}(0,T;L^2(\Omega))$ fulfills \eqref{t4f}.

Next let us prove that $u\in W^{2,r}(0,T;L^2(\Omega))$ and 
\bel{t4g}
\norm{u}_{W^{2,r}(0,T;L^2(\Omega))}\leq C\left(\norm{f}_{W^{2,r}(0,T;H^{\frac{1}{2}-\chi}(\partial\Omega))}+\norm{F}_{W^{2,r}(0,T;D(A_{\chi}^{-\frac{1}{2}}))}+\norm{u_1}_{H^{2\delta}(\Omega)}\right).
\ee
For this purpose, repeating the above arguments, we find
\bel{t4h}\begin{aligned}
u_n''(t)=&-(-1)^\chi\int_0^ts^{\alpha-1}E_{\alpha,\alpha}(-\lambda_{\chi,n}s^\alpha)\left\langle \partial_t^2f(t-s,\cdot),\tau_{\chi}^*\phi_{\chi,n}\right\rangle_{L^2(\partial\Omega)}ds\\
&+\int_0^ts^{\alpha-1}E_{\alpha,\alpha}(-\lambda_{\chi,n}s^\alpha)\left\langle \rho^{-1}\partial_t^2F(t-s,\cdot),\phi_{\chi,n}\right\rangle_{-\frac{1}{2}}ds+a_nt^{\alpha-1}E_{\alpha,\alpha}(-\lambda_{\chi,n}t^\alpha)
\end{aligned}\ee
for almost all $t\in(0,T)$, where 
$$
a_n:=-(-1)^\chi \left\langle \partial_tf(0,\cdot),\tau_{\chi}^*\phi_{\chi,n}\right\rangle_{H^{-\frac{1}{2}-\chi}(\partial\Omega),H^{\frac{1}{2}+\chi}(\partial\Omega)}+\left\langle \rho^{-1}\partial_tF(0,\cdot),\phi_{\chi,n}\right\rangle_{\frac{1}{2}}-\lambda_{\chi,n}\left\langle  u_1,\phi_{\chi,n}\right\rangle.
$$
Following the above arguments, we have
$$
\begin{aligned}
&\norm{\sum_{n=1}^\infty \left[\frac{(-1)^{\chi+1} \left\langle \partial_tf(0,\cdot),\tau_{\chi}^*\phi_{\chi,n}\right\rangle_{L^2(\partial\Omega)}+\left\langle \rho^{-1}\partial_tF(0,\cdot),\phi_{\chi,n}\right\rangle_{\frac{1}{2}}}{\lambda_{\chi,n}}\right]\phi_{\chi,n}}_{H^1(\Omega)}\\
&\leq C(\norm{\partial_tf(0,\cdot)}_{H^{\frac{1}{2}-\chi}(\partial\Omega)}+\norm{\partial_tF(0,\cdot)}_{D(A_{\chi}^{-\frac{1}{2}})})\\
\ &\leq C(\norm{f}_{W^{2,r}(0,T;H^{\frac{1}{2}-\chi}(\partial\Omega))}+\norm{F}_{W^{2,r}(0,T;D(A_{\chi}^{-\frac{1}{2}}))}).
\end{aligned}$$
By $H^{2\delta}(\Omega)=H^{2\delta}_0(\Omega)=D(A_{\chi}^\delta)$, we see that
$$\begin{aligned}
&\norm{\sum_{n=1}^\infty \left[\frac{(-1)^{\chi+1} \left\langle \partial_tf(0,\cdot),\tau_{\chi}^*\phi_{\chi,n}\right\rangle_{L^2(\partial\Omega)}+\left\langle \rho^{-1}\partial_tF(0,\cdot),\phi_{\chi,n}\right\rangle_{\frac{1}{2}}}{\lambda_{\chi,n}^{1-\delta}}\right]\phi_{\chi,n}}_{L^2(\Omega)}\\
&\leq C\left(\norm{f}_{W^{2,r}(0,T;H^{\frac{1}{2}-\chi}(\partial\Omega))}+\norm{F}_{W^{2,r}(0,T;D(A_{\chi}^{-\frac{1}{2}}))}\right).
\end{aligned}$$
Therefore, using (1.148) of \cite[Theorem 1.6]{P}, we deduce that
$$\begin{aligned}
&\norm{\sum_{n=1}^\infty \left[a_nt^{\alpha-1}E_{\alpha,\alpha}(-\lambda_{\chi,n}t^\alpha)\right]\phi_{\chi,n}}_{L^2(\Omega)}\\
&\leq C\left(\norm{f}_{W^{2,r}(0,T;H^{\frac{1}{2}-\chi}(\partial\Omega))}+\norm{F}_{W^{2,r}(0,T;D(A_{\chi}^{-\frac{1}{2}}))}+\norm{u_1}_{H^{2\delta}(\Omega)}\right)t^{\delta\alpha-1}
\end{aligned}$$
for almost all $t \in (0,T)$.
Taking into account condition \eqref{t4a}, we obtain
$$\begin{aligned}
&\norm{\sum_{n=1}^\infty \left[a_nt^{\alpha-1}E_{\alpha,\alpha}(-\lambda_{\chi,n}t^\alpha)\right]\phi_{\chi,n}}_{L^r(0,T;L^2(\Omega))}\\
&\leq CT^{\delta\alpha-1+\frac{1}{r}}\left(\norm{f}_{W^{2,r}(0,T;H^{\frac{1}{2}-\chi}(\partial\Omega))}+\norm{F}_{W^{2,r}(0,T;D(A_{\chi}^{-\frac{1}{2}}))}+\norm{u_1}_{H^{2\delta}(\Omega)}\right).
\end{aligned}$$
Repeating the arguments used for proving $u\in L^r(0,T;L^2(\Omega))$, 
we see that the sequences in $N$: 
$$
\sum_{n=1}^N-(-1)^\chi\left(\int_0^ts^{\alpha-1}E_{\alpha,\alpha}(-\lambda_{\chi,n}s^\alpha)\left\langle \partial_t^2f(t-s,\cdot),\tau_{\chi}^*\phi_{\chi,n}
\right\rangle_{L^2(\partial\Omega)}ds\right) \phi_{\chi,n} 
$$
and
$$
\sum_{n=1}^N \left(\int_0^ts^{\alpha-1}E_{\alpha,\alpha}(-\lambda_{\chi,n}s^\alpha)\left\langle \rho^{-1}\partial_t^2F(t-s,\cdot),\phi_{\chi,n}\right\rangle_{-\frac{1}{2}}ds\right)\phi_{\chi,n},
$$
converge in $L^r(0,T;L^2(\Omega))$ and we have
$$\begin{aligned}
&\norm{\sum_{n=1}^\infty-(-1)^\chi\left(\int_0^ts^{\alpha-1}E_{\alpha,\alpha}
(-\lambda_{\chi,n}s^\alpha)\left\langle \partial_t^2f(t-s,\cdot),\tau_{\chi}^*
\phi_{\chi,n}\right\rangle_{L^2(\partial\Omega)}ds\right)\phi_{\chi,n}}
_{L^r(0,T;L^2(\Omega))}\\
& +\norm{\sum_{n=1}^\infty\left(\int_0^ts^{\alpha-1}E_{\alpha,\alpha}(-\lambda_{\chi,n}s^\alpha)\left\langle \rho^{-1}\partial_t^2F(t-s,\cdot),\phi_{\chi,n}\right\rangle_{-\frac{1}{2}}ds\right)\phi_{\chi,n}}_{L^r(0,T;L^2(\Omega))}\\
&\leq C\left(\norm{f}_{W^{2,r}(0,T;H^{\frac{1}{2}-\chi}(\partial\Omega))}+\norm{F}_{W^{2,r}(0,T;D(A_{\chi}^{-\frac{1}{2}}))}\right).
\end{aligned}$$
Therefore, we have $u\in W^{2,r}(0,T;L^2(\Omega))$ and combining these two last estimates with \eqref{t4f} and \eqref{t4h}, we reach \eqref{t4g}. 

Next we prove that $u\in L^r(0,T;H^2(\Omega))$ is a strong solution to
\eqref{eq1} and 
\bel{t4i}\begin{aligned}
&\norm{u}_{L^r(0,T;H^2(\Omega))}\\
&\leq C\left(\norm{f}_{W^{2,r}(0,T;H^{\frac{1}{2}-\chi}(\partial\Omega))}+\norm{F}_{W^{2,r}(0,T;D(A_{\chi}^{-\frac{1}{2}}))}+\norm{u_1}_{H^{2\delta}(\Omega)}\right)\\
&+ \ \ \ C\left(\norm{f}_{L^r(0,T;H^{\frac{3}{2}-\chi}(\partial\Omega))}+\norm{F}_{L^r(0,T;L^2(\Omega))}\right).
\end{aligned}
\ee
For this purpose, in a similar way to Theorem \ref{t3}, let 
$$
\widetilde{f}\in W^{2,r}(0,+\infty;H^{\frac{1}{2}-\chi}(\pd\Omega))
\cap L^r(0,+\infty;H^{\frac{3}{2}-\chi}(\pd\Omega))
$$
and
$$
\widetilde{F}\in W^{2,r}(0,+\infty;D(A_{\chi}^{-\frac{1}{2}}))\cap L^r(0,+\infty;L^2(\Omega))
$$
satisfy supp $\widetilde{f}\subset[0,T+1)\times\pd \Omega$, 
supp $\widetilde{F} \subset[0,T+1)\times\overline{ \Omega}$ and \eqref{t3k}.
We set 
$$
\widetilde{v}(t,\cdot):=\sum_{n=1}^\infty \widetilde{v}_n(t)\phi_{\chi,n}
$$
where  
$$\begin{aligned}
\widetilde{v}_n(t)=&-(-1)^\chi\int_0^t(t-s)^{\alpha-1}E_{\alpha,\alpha}(-\lambda_{\chi,n}(t-s)^\alpha)\left\langle \widetilde{f}(s,\cdot),\tau_{\chi}^*\phi_{\chi,n}\right\rangle_{L^2(\partial\Omega)}ds\\
\ &+\int_0^t(t-s)^{\alpha-1}E_{\alpha,\alpha}(-\lambda_{\chi,n}(t-s)^\alpha)\left\langle \rho^{-1}\widetilde{F}(s,\cdot),\phi_{\chi,n}\right\rangle ds
+u_{2,n}(t)+u_{3,n}(t),\quad t>0.
\end{aligned}$$
From the above arguments, one can verify that $\widetilde{v}\in 
W^{2,r}_{loc}(0,+\infty;L^2(\Omega))$ and  
\bel{t4jj}
e^{-pt}\widetilde{v}(t,\cdot)\in W^{2,1}(0,+\infty;L^2(\Omega))
\ee
for $p>0$.  We assert 
\bel{t4j}
e^{-pt}\widetilde{v}(t,\cdot)\in L^1(0,+\infty;H^1(\Omega)),\quad p>0.
\ee
{\bf Proof of \eqref{t4j}.}
In a similar way to the proof of Theorem \ref{t3}, integrating by parts and using the compatibility condition \eqref{com1}, we obtain
$$\begin{aligned}
\widetilde{v}_n(t)=&(-1)^\chi\int_0^tE_{\alpha,1}(-\lambda_{\chi,n}s^\alpha)\frac{\left\langle \partial_t\widetilde{f}(t-s,\cdot),\tau_{\chi}^*\phi_{\chi,n}\right\rangle_{H^{-\frac{1}{2}-\chi}(\partial\Omega),H^{\frac{1}{2}+\chi}(\partial\Omega)}}{\lambda_{\chi,n}}ds\\
&-(-1)^\chi\frac{\left\langle \widetilde{f}(t,\cdot),\tau_{\chi}^*\phi_{\chi,n}\right\rangle_{L^2(\partial\Omega)}}{\lambda_{\chi,n}}+u_{3,n}(t)\\
&-\int_0^tE_{\alpha,1}(-\lambda_{\chi,n}s^\alpha)\frac{\left\langle \rho^{-1}\partial_t\widetilde{F}(t-s,\cdot),\phi_{\chi,n}\right\rangle_{-\frac{1}{2}}}{\lambda_{\chi,n}}ds
+ \frac{\left\langle \rho^{-1}\widetilde{F}(t,\cdot),\phi_{\chi,n}\right\rangle}{\lambda_{\chi,n}}.
\end{aligned}$$
Since $\widetilde{f}\in L^1(0,+\infty;H^{\frac{3}{2}-\chi}
(\pd\Omega))$, $\widetilde{F}\in L^1(0,+\infty;L^2(\Omega))$ 
and $u_0\in L^2(\Omega)$, one can verify 
$$\begin{aligned}
&\norm{e^{-pt}\sum_{n=1}^\infty \left(-(-1)^\chi\frac{\left\langle \widetilde{f}(t,\cdot),\tau_{\chi}^*\phi_{\chi,n}\right\rangle_{L^2(\partial\Omega)}}{\lambda_{\chi,n}}+u_{3,n}(t)\right)\phi_{\chi,n}}_{L^1(0,+\infty;H^1(\Omega))}\\
&+\norm{e^{-pt}\frac{\left\langle \rho^{-1}\widetilde{F}(t,\cdot),\phi_{\chi,n}\right\rangle}{\lambda_{\chi,n}}}_{L^1(0,+\infty;H^1(\Omega))}\\
&\leq C\left(\norm{\widetilde{f}}_{L^1(0,+\infty;H^{\frac{3}{2}-\chi}(\pd\Omega))}+\norm{\widetilde{F}}_{L^1(0,+\infty;L^2(\Omega))}
+ \left(\int_0^\infty e^{-pt}t^{1-\alpha}dt\right)\norm{u_1}_{L^2(\Omega)}\right)<\infty
\end{aligned}$$
for all $p>0$.
Therefore, \eqref{t4j} will be proved if we show  
\bel{t4k}
\norm{e^{-pt}\sum_{n=1}^\infty(-1)^\chi\left(
\int_0^tE_{\alpha,1}(-\lambda_{\chi,n}s^\alpha)\frac{\left\langle \partial_t\widetilde{f}(t-s,\cdot),\tau_{\chi}^*\phi_{\chi,n}\right\rangle}
{\lambda_{\chi,n}}ds\right) \phi_{\chi,n}}_{L^1(0,+\infty;H^1(\Omega))}<\infty
\ee
and
\bel{t4kk}
\norm{e^{-pt}\sum_{n=1}^\infty(-1)^\chi\left(
\int_0^tE_{\alpha,1}(-\lambda_{\chi,n}s^\alpha)\frac{\left\langle \rho^{-1}\partial_t\widetilde{F}(t-s,\cdot),\phi_{\chi,n}\right\rangle_{\frac{1}{2}}}{\lambda_{\chi,n}}ds\right) 
\phi_{\chi,n}}_{L^1(0,+\infty;H^1(\Omega))}<\infty
\ee
for all $p>0$.
Applying formula (1.148) in \cite[Theorem 1.6]{P}, we obtain
$$\begin{aligned}
&\norm{\sum_{n=1}^\infty(-1)^\chi\left( \int_0^tE_{\alpha,1}(-\lambda_{\chi,n}s^\alpha)\frac{\left\langle \partial_t\widetilde{f}(t-s,\cdot),\tau_{\chi}^*\phi_{\chi,n}\right\rangle_{L^2(\partial\Omega)}}{\lambda_{\chi,n}}ds\right)
\phi_{\chi,n}}_{H^1(\Omega)}\\
&\leq C\norm{\sum_{n=1}^\infty(-1)^\chi\left(
\int_0^tE_{\alpha,1}(-\lambda_{\chi,n}s^\alpha)\frac{\left\langle \partial_t\widetilde{f}(t-s,\cdot),\tau_{\chi}^*\phi_{\chi,n}\right\rangle_{L^2(\partial\Omega)}}{\lambda_{\chi,n}}ds\right)\phi_{\chi,n}}_{D(A_{\chi}^{\frac{1}{2}})}\\
&\leq C\int_0^t(t-s)^{-\frac{\alpha}{2}}\norm{\sum_{n=1}^\infty\frac{\left\langle \partial_t\widetilde{f}(t-s,\cdot),\tau_{\chi}^*\phi_{\chi,n}\right\rangle_{L^2(\partial\Omega)}}{\lambda_{\chi,n}}\phi_{\chi,n}}_{L^2(\Omega)}ds\\
&\leq C\int_0^t(t-s)^{-\frac{\alpha}{2}}\norm{\partial_t\widetilde{f}(s,\cdot)}_{H^{\frac{1}{2}-\chi}(\pd\Omega)}ds
\end{aligned}$$
for almost all $t>0$. 
It follows that 
$$\begin{aligned}
&\norm{e^{-pt}\sum_{n=1}^\infty(-1)^\chi\left(
\int_0^tE_{\alpha,1}(-\lambda_{\chi,n}s^\alpha)\frac{\left\langle \partial_t\widetilde{f}(t-s,\cdot),\tau_{\chi}^*\phi_{\chi,n}\right\rangle_{L^2(\partial\Omega)}}{\lambda_{\chi,n}}ds\right)
\phi_{\chi,n}}_{L^1(0,+\infty;H^1(\Omega))}\\
&\leq C\norm{(e^{-pt}t^{-\frac{\alpha}{2}}\mathds{1}_{(0,+\infty)})
*\left(\norm{\partial_t\widetilde{f}(t,\cdot)}_{H^{\frac{1}{2}-\chi}
(\pd\Omega)}\mathds{1}_{(0,+\infty)}\right)}_{L^1(0,+\infty)}\\
&\leq C_p\norm{\widetilde{f}}_{W^{1,r}(0,+\infty;H^{\frac{1}{2}-\chi}(\pd\Omega))} <\infty
\end{aligned}
$$
for all $p>0$.  This proves \eqref{t4k}. 
Moreover we can verify \eqref{t4kk} in the same way.   
Thus the proof of \eqref{t4j} is complete.\qed
 

By \eqref{t4jj} and \eqref{t4j}, the Laplace transform 
$\mathcal{L} \widetilde{v}(p,\cdot)$ of $\widetilde{v}$ satisfies 
$$
p^\alpha(\mathcal{L} \widetilde{v}(p,x) - p^{-1}u_0(x)
- p^{-2}u_1(x)) + \rho^{-1}\mathcal A \mathcal{L} \widetilde{v}(p,x)
= \int_0^{T+1}e^{-pt}\rho^{-1}\widetilde{F}(t,x)dt
$$
for $(p,x)\in (0, +\infty) \times\Omega$.

Repeating the corresponding arguments for Theorem \ref{t3}, we deduce that
$$
\left\langle (\rho\partial_t^\alpha \widetilde{v}+\mathcal A \widetilde{v})(t,\cdot)-\widetilde{F}(t,\cdot),\psi\right\rangle_{D'(\Omega),\mathcal C^\infty_0(\Omega)}=0
$$
for almost all $t\in (0, +\infty)$.
On the other hand, in view of \eqref{t3k} one can easily verify that 
$\widetilde{v}=u$ in $(0,T)\times\Omega$, which implies that 
$$
\left\langle (\rho\partial_t^\alpha u+\mathcal A u)(t,\cdot)-F(t,\cdot),\psi\right\rangle_{D'(\Omega),\mathcal C^\infty_0(\Omega)}=0,\quad \psi\in\mathcal C^\infty_0(\Omega)
$$
for almost all $t\in(0,T)$.

Hence, since $\mathcal{L} \widetilde{v}(p,\cdot)$ satisfies
$$
\mathcal{L} \widetilde{v}(p,x) = \mathcal L \widetilde{f}(p,x),
\quad x\in\pd\Omega
$$
for all $p>0$, it follows from \eqref{t4j} that 
$u(t,\cdot)$ solves the boundary value problem
$$
\left\{ \begin{array}{rcll} 
\rho(x)^{-1} \cA u(t,x)& = & -\partial_t^\alpha u(t,x)+F(t,x), & x\in 
\Omega,\\
\tau_{\chi} u(t,x) & = & f(t,x), & x \in  \pd \Omega 
\end{array}
\right.$$
for almost all $t\in(0,T)$.
Therefore, by $u\in W^{2,r}(0,T;L^2(\Omega)))$, we obtain 
$\partial_t^\alpha u\in L^r(0,T;L^2(\Omega))$ and, combining this with 
$f\in L^r(0,T;H^{\frac{3}{2}-\chi}(\pd\Omega))$ and 
$F\in L^r(0,T;L^2(\Omega))$, we see that $u\in L^r(0,T;H^2(\Omega))$. From the above properties, one can easily verify that $u$ is a strong solution 
to \eqref{eq1} in the sense of Definition \ref{d3}. Moreover, applying Young's 
inequality, we see that $u$  satisfies
$$\begin{aligned}\norm{u}_{L^r(0,T;H^2(\Omega))}&\leq C\left(\norm{f}_{L^r(0,T;H^{\frac{3}{2}-\chi}(\pd\Omega))}+\norm{F}_{L^r(0,T;L^2(\Omega))}+\norm{\partial_t^\alpha u}_{L^r(0,T;L^2(\Omega))}\right)\\
\ &\leq  C\left(\norm{f}_{L^r(0,T;H^{\frac{3}{2}-\chi}(\pd\Omega))}+\norm{F}_{L^r(0,T;L^2(\Omega))}+\norm{ u}_{W^{2,r}(0,T;L^2(\Omega))}\right).\end{aligned}$$
Combining this with \eqref{t4g}, we obtain \eqref{t4i}. Finally, using arguments similar to those used for proving \eqref{t4g}, we deduce that $u\in W^{1,r}(0,T;H^1(\Omega))$ satisfies
$$\norm{u}_{W^{1,r}(0,T;H^1(\Omega))}\leq C\left(\norm{f}_{W^{2,r}(0,T;H^{\frac{1}{2}-\chi}(\partial\Omega))}+\norm{F}_{W^{2,r}(0,T;D(A_{\chi}^{-\frac{1}{2}}))}+\norm{u_1}_{H^{2\delta}(\Omega)}\right).$$
Combining \eqref{t4g} and \eqref{t4i} with the above estimate, we reach 
\eqref{t4b}.  Thus this completes the proof of Theorem \ref{t4}.

\section{Proof of Theorem \ref{t5}}

We will consider only the case $m=3$, because the case $m\geq4$ can be 
proved by an iteration argument. Let $u$ be the weak solution to \eqref{eq1}. 
Similarly to the previous section, one can verify that 
$u\in W^{1,r}(0,T;H^1(\Omega))$ and 
$$
\norm{u}_{W^{1,r}(0,T;H^1(\Omega))}\leq C\left(\norm{f}_{W^{2,r}(0,T;H^{\frac{1}{2}-\chi}(\partial\Omega))}+\norm{F}_{W^{2,r}(0,T;D(A_{\chi}^{-\frac{1}{2}}))}
\right).
$$
Moreover, condition \eqref{com2} implies that 
$$
\lambda_{\chi,n}\left\langle  u_1,\phi_{\chi,n}\right\rangle=-(-1)^\chi \left\langle \partial_tf(0,\cdot),\tau_{\chi}^*\phi_{\chi,n}\right\rangle_{L^2(\partial\Omega)}+\left\langle \rho^{-1}\partial_tF(0,\cdot),\phi_{\chi,n}\right\rangle.$$
Combining this with \eqref{t4h}, for almost all $t\in(0,T)$, we obtain
\bel{t5d}\begin{aligned}
u_n''(t)=&-(-1)^\chi\int_0^ts^{\alpha-1}E_{\alpha,\alpha}(-\lambda_{\chi,n}s^\alpha)\left\langle \partial_t^2f(t-s,\cdot),\tau_{\chi}^*\phi_{\chi,n}\right\rangle_{L^2(\partial\Omega)}ds\\
\ &+\int_0^ts^{\alpha-1}E_{\alpha,\alpha}(-\lambda_{\chi,n}s^\alpha)\left\langle \rho^{-1}\partial_t^2F(t-s,\cdot),\phi_{\chi,n}\right\rangle ds.
\end{aligned}
\ee
Similarly to the proof of Theorem \ref{t1}, we can see that $w=\partial_t^2u$ 
is the unique weak solution to 
$$\left\{ \begin{array}{rcll} 
(\rho(x) \partial_t^{\alpha}+\cA ) w(t,x) & = & \partial_t^2F(t,x), & (t,x)\in
(0,T) \times \Omega,\\
\tau_{\chi} w(t,x) & = & \partial_t^2f(t,x), & (t,x) \in 
(0,T) \times \pd \Omega, \\  
\pd_t^k w(0,x) & = & 0, & x \in \Omega,\ k=0,1.
\end{array}
\right.
$$
Since $\partial_t^2F\in L^r(0,T;L^2(\Omega))\cap W^{1,r}(0,T;D(A_{\chi}
^{-\frac{1}{2}}))$ and $\partial_t^2f\in  W^{1,r}(0,T;H^{\frac{1}{2}-\chi}
(\pd\Omega))$,  following the proof of Theorem \ref{t4}, we can prove that 
$\partial_t^2u\in L^r(0,T;H^1(\Omega))$ and 
$$\begin{aligned}
\norm{\partial_t^2u}_{L^r(0,T; H^{1}(\Omega))}\leq& C \left(\norm{\pd_t^2f}_{W^{1,r}(0,T;H^{\frac{1}{2}-\chi}(\pd\Omega))}+\norm{\partial_t^2F}
_{L^r(0,T;L^2(\Omega))}\right)\\
\ &+C \norm{\partial_t^2F}_{W^{1,r}(0,T; D(A_{\chi}^{-\frac{1}{2}}))}.
\end{aligned}$$
It follows that $u\in W^{2,r}(0,T; H^1(\Omega))$ satisfies
\bel{t5g}
\begin{aligned}\norm{u}_{W^{2,r}(0,T; H^{1}(\Omega))}\leq& C \left(\sum_{k=2}^3\norm{f}_{W^{3-k,r}(0,T; H^{k-\frac{1}{2}-\chi}(\pd\Omega))}+\norm{F}_{W^{3-k,r}(0,T; H^{k-2}(\Omega))}\right)\\
\ &+C \left(\norm{f}_{W^{3,r}(0,T; H^{\frac{1}{2}-\chi}(\pd\Omega))}+\norm{F}_{W^{3,r}(0,T; D(A_{\chi}^{-\frac{1}{2}}))}\right).
\end{aligned}\ee
Repeating the arguments used for Theorem \ref{t4}, from these properties
we obtain that $u$ is a strong solution to \eqref{eq1} and $u(t,\cdot)$ 
solves the boundary value problem
$$
\left\{ \begin{array}{rcll} 
\rho(x)^{-1} \cA u(t,x)& = & -\partial_t^\alpha u(t,x), & x\in   \Omega,\\
\tau_{\chi} u(t,x) & = & f(t,x), & x \in  \pd \Omega 
\end{array}
\right.
$$
for almost all $t\in(0,T)$.  
Therefore, by means of $u\in W^{2,r}(0,T;H^1(\Omega))$, we deduce that $\partial_t^\alpha u\in L^r(0,T;H^1(\Omega))$ and, combining this with 
$f\in L^r(0,T;H^{\frac{5}{2}-\chi}(\pd\Omega))$ and 
$F\in L^r(0,T;H^1(\Omega))$, we see that $u\in L^r(0,T;H^3(\Omega))$. 
Furthermore $u$  satisfies
$$\begin{aligned}
\norm{u}_{L^r(0,T;H^3(\Omega))}&\leq C\left(\norm{f}_{L^r(0,T;H^{\frac{5}{2}-\chi}(\pd\Omega))}+\norm{F}_{L^r(0,T;H^1(\Omega))}+\norm{\partial_t^\alpha u}_{L^r(0,T;H^1(\Omega))}\right)\\
\ &\leq  C\left(\norm{f}_{L^r(0,T;H^{\frac{5}{2}-\chi}(\pd\Omega))}+\norm{F}_{L^r(0,T;H^1(\Omega))}+\norm{ u}_{W^{2,r}(0,T;H^1(\Omega))}\right)
\end{aligned}$$
and, applying \eqref{t5g}, we obtain
$$\begin{aligned}
\norm{u}_{L^r(0,T; H^3(\Omega))}\leq& C \left(\sum_{k=2}^3\norm{f}
_{W^{3-k,r}(0,T; H^{k-\frac{1}{2}-\chi}(\pd\Omega))}+\norm{F}_{W^{3-k,r}(0,T; H^{k-2}(\Omega))}\right)\\
\ &+C \left(\norm{f}_{W^{3,r}(0,T; H^{\frac{1}{2}-\chi}(\pd\Omega))}+\norm{F}_{W^{3,r}(0,T; D(A_{\chi}^{-\frac{1}{2}}))}\right).
\end{aligned}$$
In the same way, we prove that $u\in W^{3,r}(0,T;L^2(\Omega))\cap 
W^{2,r}(0,T;H^1(\Omega))\cap W^{1,r}(0,T;H^2(\Omega))$ and $u$ satisfies 
\eqref{t5b}.   Thus this completes the proof of Theorem \ref{t5}.

\section{Proof of Propositions \ref{p1} and \ref{p2}}

\textbf{Proof of Proposition \ref{p1}.} We start with the first statement of 
Proposition \ref{p1}. For this purpose, we fix 
$f\in\mathcal C^\infty([0,T]\times \partial\Omega)$, $F\in\mathcal 
C^\infty([0,T]\times \overline{\Omega})$ and $u_0\in\mathcal C^\infty(\overline{\Omega})$ which do not satisfy 
\eqref{com1}.  We assume also that \eqref{eq1} admits a unique solution $u\in W^{1,1}(0,T;D(A_{\chi}^{-\frac{1}{2}}))$.
In view of \eqref{eleq1}, there exists $n_0\in\mathbb N$ such that 
\bel{r1a}b_0:=-(-1)^\chi\left\langle f(0,\cdot),\tau_{\chi}^*
\phi_{\chi,n_0}\right\rangle_{L^2(\partial\Omega)}
+ \left\langle \rho^{-1}F(0,\cdot),\phi_{\chi,n_0}\right\rangle 
-\lambda_{\chi,n_0}\left\langle u_0,\phi_{\chi,n_0}\right\rangle\neq0.
\ee 
We will prove that $u\notin W^{1,(1-\alpha)^{-1}}(0,T;D(A_{\chi}
^{-\frac{1}{2}}))$. 
Following the argumentation of Theorem \ref{t3}, we can show that
$$\begin{aligned}
&\left\langle \partial_tu(t,\cdot),\phi_{\chi,n_0}\right\rangle
_{-\frac{1}{2}}\\
&= -(-1)^\chi\int_0^ts^{\alpha-1}E_{\alpha,\alpha}(-\lambda_{\chi,n_0}s^\alpha)
\left\langle \partial_tf(t-s,\cdot),\tau_{\chi}^*\phi_{\chi,n_0}
\right\rangle_{L^2(\partial\Omega)}ds\\
&\ \ \ 
+ \int_0^ts^{\alpha-1}E_{\alpha,\alpha}(-\lambda_{\chi,n_0}s^\alpha)
\left\langle \rho^{-1}\partial_tF(t-s,\cdot),\phi_{\chi,n_0}
\right\rangle ds\\
&\ \ \ + b_0 t^{\alpha-1}E_{\alpha,\alpha}(-\lambda_{\chi,n_0}t^\alpha).
\end{aligned}$$
On the other hand, since $E_{\alpha,\alpha}(-\lambda_{\chi,n_0}t^\alpha)\in
\mathcal C[0,T]$ and $E_{\alpha,\alpha}(0)=\frac{1}{\Gamma(\alpha)}>0$, 
we see that there exists $t_0\in(0,T)$ and $c_0>0$ such that 
$$
\inf_{t\in[0,t_0]}|E_{\alpha,\alpha}(-\lambda_{\chi,n}t^\alpha)|=c_0.
$$
Therefore, for almost all $t\in(0,t_0)$, we have
$$\begin{aligned}
& |\left\langle \partial_tu(t,\cdot),\phi_{\chi,n_0}\right\rangle
_{-\frac{1}{2}}|\\
\geq& c_0\abs{b_0} t^{\alpha-1}-\abs{-(-1)^\chi\int_0^ts^{\alpha-1}
E_{\alpha,\alpha}(-\lambda_{\chi,n_0}s^\alpha)\left\langle \partial_tf(t-s,\cdot),\tau_{\chi}^*\phi_{\chi,n_0}\right\rangle_{L^2(\partial\Omega)}ds}\\
&-\abs{\int_0^ts^{\alpha-1}E_{\alpha,\alpha}(-\lambda_{\chi,n_0}s^\alpha)\left\langle \rho^{-1}\partial_tF(t-s,\cdot),\phi_{\chi,n_0}\right\rangle ds}.
\end{aligned}$$
Moreover, since $f\in\mathcal C^\infty([0,T]\times \partial\Omega)$, 
$F\in\mathcal C^\infty([0,T]\times \overline{\Omega})$, it is clear that 
$$\begin{aligned}
c_1 :=&\norm{-(-1)^\chi\int_0^ts^{\alpha-1}E_{\alpha,\alpha}
(-\lambda_{\chi,n_0}s^\alpha)\left\langle \partial_tf(t-s,\cdot),\tau_{\chi}^*
\phi_{\chi,n_0}\right\rangle_{L^2(\partial\Omega)}ds}_{L^\infty(0,T)}\\
& + \norm{\int_0^ts^{\alpha-1}E_{\alpha,\alpha}(-\lambda_{\chi,n_0}s^\alpha)
\left\langle \rho^{-1}\partial_tF(t-s,\cdot),\phi_{\chi,n_0}\right\rangle ds}
_{L^\infty(0,T)}<\infty.
\end{aligned}$$
Thus, for almost all $t\in(0,t_0)$, we have
$$
|\left\langle \partial_tu(t,\cdot),\phi_{\chi,n_0}\right\rangle_{-\frac{1}{2}}|
\geq c_0\abs{b_0} t^{\alpha-1} - c_1
$$
and condition \eqref{r1a} clearly implies that
$$
|\left\langle \partial_tu(t,\cdot),\phi_{\chi,n_0}\right\rangle|
\notin L^{(1-\alpha)^{-1}}(0,T).
$$
Thus, we have $u\notin W^{1,(1-\alpha)^{-1}}(0,T;D(A_{\chi}^{-\frac{1}{2}}))$,
which completes the proof of the first statement of Proposition \ref{p1}.

For the second statement of Proposition \ref{p1}, let us consider 
$f\in\mathcal C^\infty([0,T]\times \partial\Omega)$, 
$F\in\mathcal C^\infty([0,T] \times \overline{\Omega})$ 
and $u_0\in\mathcal C^\infty(\overline{\Omega})$ satisfying \eqref{com1} 
and let \eqref{eq1} admit a unique solution $u\in W^{2,1}(0,T;D(A_{\chi}^{-\frac{1}{2}}))$. Assume also that \eqref{tt3a} is not fulfilled for $m=2$. 
Then we first show that there exists $n_1\in\mathbb N$ such that
\bel{r2c}
-(-1)^\chi\left\langle \partial_tf(0,\cdot),\tau_{\chi}^*\phi_{\chi,n_1}
\right\rangle_{L^2(\partial\Omega)}
+ \left\langle \rho^{-1}\partial_tF(0,\cdot),\phi_{\chi,n_1}\right\rangle
\neq 0.\ee
Indeed, assuming the contrary, we deduce that
$$\lambda_{\chi,k}\left\langle A_{\chi}^{-1}\rho^{-1}\partial_tF(0,\cdot),\phi_{\chi,k}\right\rangle=-(-1)^\chi\left\langle -\partial_tf(0,\cdot),\tau_{\chi}^*\phi_{\chi,k}\right\rangle_{L^2(\partial\Omega)},\quad k\in\mathbb N.$$
Then, following formula \eqref{eleq1}, we deduce that $G=A_{\chi}^{-1}\rho^{-1}\partial_tF(0,\cdot)$ solves the boundary value problem
$$\left\{ \begin{array}{rcll} 
 \cA G(x)& = & 0, & x\in   \Omega,\\
\tau_{\chi} G(x) & = & -\partial_tf(0,x), & x \in  \pd \Omega. 
\end{array}
\right.
$$
On the other hand, since $\rho^{-1}\partial_tF(0,\cdot)\in L^2(\Omega)$
and $G=A_{\chi}^{-1}\rho^{-1}\partial_tF(0,\cdot)$, we obtain 
$\partial_tF(0,\cdot) =\rho A_{\chi} G=\cA G=0$ and $\partial_tf(0,\cdot)
= -\tau_{\chi} G=0$. This contradicts that \eqref{tt3a} is not fulfilled 
for $m=2$. Thus, there exists $n_1\in\mathbb N$ such that \eqref{r2c} 
holds true. Repeating the arguments for Theorem \ref{t3}, in view of 
\eqref{com1}, one can verify that
$$\begin{aligned}
&\left\langle \partial_t^2u(t,\cdot),\phi_{\chi,n_1}\right\rangle_{-\frac{1}{2}}\\
&=-(-1)^\chi\int_0^ts^{\alpha-1}E_{\alpha,\alpha}(-\lambda_{\chi,n_1}
s^\alpha)\left\langle \partial_t^2f(t-s,\cdot),\tau_{\chi}^*\phi_{\chi,n_1}
\right\rangle_{L^2(\partial\Omega)}ds\\
&\ \ \ +\int_0^ts^{\alpha-1}E_{\alpha,\alpha}(-\lambda_{\chi,n_1}
s^\alpha)\left\langle \rho^{-1}\partial_t^2F(t-s,\cdot),\phi_{\chi,n_1}
\right\rangle ds\\
&\ \ \ + [-(-1)^\chi\left\langle \partial_tf(0,\cdot),\tau_{\chi}^*
\phi_{\chi,n_1}\right\rangle_{L^2(\partial\Omega)}+\left\langle \rho^{-1}\partial_tF(0,\cdot),\phi_{\chi,n_1}\right\rangle] t^{\alpha-1}E_{\alpha,\alpha}
(-\lambda_{\chi,n_1}t^\alpha).
\end{aligned}$$
Therefore, similarly to the proof of the first statement of Proposition 
\ref{p1},  by \eqref{r2c} we deduce that
$\left\langle \partial_t^2u(t,\cdot),\phi_{\chi,n_1}\right\rangle
_{-\frac{1}{2}}\notin L^{(\alpha-1)^{-1}}(0,T)$, which implies that $\partial_t^2u\notin L^{(\alpha-1)^{-1}}(0,T; D(A_{\chi}^{-\frac{1}{2}}))$. 
Thus the proof of the proposition is comppleted.
\qed
\\

\textbf{Proof of Proposition \ref{p2}.} 
We start with the proof of the first statement.
For this purpose, we fix $f\in\mathcal C^\infty([0,T]\times \partial\Omega)$,
$F\in\mathcal C^\infty([0,T]\times \overline{\Omega})$ and 
$u_0, u_1\in\mathcal C^\infty(\overline{\Omega})$ such that  
\eqref{com1} is not fulfilled. Then, there exists $n_0\in \mathbb N$ 
such that \eqref{r1a} is not fulfilled. 

Let us assume that the solution $u$ to \eqref{eq1} is lying in $ W^{2,1}(0,T;D(A_{\chi}^{-\frac{1}{2}}))$. We will prove that $u\notin W^{2,(2-\alpha)^{-1}}(0,T;D(A_{\chi}^{-\frac{1}{2}}))$. Combining \cite[Lemma 3.2]{SY} 
with Proposition \ref{p1}, we deduce that
$$\begin{aligned}
&\left\langle \partial_t^2u(t,\cdot),\phi_{\chi,n_0}\right\rangle
_{-\frac{1}{2}}\\
&= -(-1)^\chi\int_0^ts^{\alpha-1}E_{\alpha,\alpha}(-\lambda_{\chi,n_0}
s^\alpha)\left\langle \partial_t^2f(t-s,\cdot),\tau_{\chi}^*\phi_{\chi,n_0}
\right\rangle_{L^2(\partial\Omega)}ds\\
&\ \ \ +\int_0^ts^{\alpha-1}E_{\alpha,\alpha}(-\lambda_{\chi,n_0}s^\alpha)
\left\langle \rho^{-1}\partial_t^2F(t-s,\cdot),\phi_{\chi,n_0}
\right\rangle ds\\
&\ \ \ + d_0t^{\alpha-2}E_{\alpha,\alpha-1}(-\lambda_{\chi,n_0}t^\alpha) 
+ e_0t^{\alpha-1}E_{\alpha,\alpha}(-\lambda_{\chi,n_0}t^\alpha),
\end{aligned}$$
where 
$$
d_0=-(-1)^\chi \left\langle f(0,\cdot),\tau_{\chi}^*\phi_{\chi,n_0}
\right\rangle_{L^2(\partial\Omega)}
+ \left\langle \rho^{-1}F(0,\cdot),\phi_{\chi,n_0} \right\rangle
-\lambda_{\chi,n_0}\left\langle  u_0,\phi_{\chi,n_0}\right\rangle
$$
and
$$
e_0=-(-1)^\chi \left\langle \partial_tf(0,\cdot),\tau_{\chi}^*\phi_{\chi,n_0}
\right\rangle_{L^2(\partial\Omega)}
+ \left\langle \rho^{-1}\partial_tF(0,\cdot),\phi_{\chi,n_0} \right\rangle
- \lambda_{\chi,n_0}\left\langle  u_1,\phi_{\chi,n_0}\right\rangle.
$$
Therefore, in a similar way to Proposition \ref{p1}, one can show that there 
exists $t_1\in(0,T)$, $c_2, c_3>0$ such that
$$
\abs{\left\langle \partial_t^2u(t,\cdot),\phi_{\chi,n_0}\right\rangle
_{-\frac{1}{2}}}\geq c_3\abs{d_0}t^{\alpha-2}-c_2,\quad t\in(0,t_1),
$$
 which proves that  $u\notin W^{2,(2-\alpha)^{-1}}(0,T;D(A_{\chi}^{-\frac{1}{2}}))$. The rest part of the proposition can be proved similarly and is omitted.
\qed

\section*{Acknowledgments}
This work was supported by Grant-in-Aid for Scientific Research (S)
The first author is partially supported by  the French National
Research Agency ANR (project MultiOnde) grant ANR-17-CE40-0029.
The second author is partly supported by  
by The National Natural Science Foundation of China 
(no. 11771270, 91730303).
This work was prepared with the support of the "RUDN University Program 5-100".

 
\end{document}